\documentclass[draft,reqno]{amsart}
\usepackage{amssymb,amsmath}
\usepackage{verbatim}
\usepackage{color}
\usepackage[latin1]{inputenc}

\newtheorem{theorem}{Theorem}
\newtheorem{lemma}{Lemma}
\newtheorem{proposition}{Proposition}
\newtheorem{corollary}{Corollary}

\theoremstyle{definition}
\newtheorem{definition}{Definition}
\newtheorem{example}{Example}

\theoremstyle{remark}
\newtheorem{remark}{Remark}

\numberwithin{equation}{section}
\numberwithin{theorem}{section}
\numberwithin{lemma}{section}
\numberwithin{proposition}{section}
\numberwithin{corollary}{section}
\numberwithin{definition}{section}
\numberwithin{example}{section}
\numberwithin{remark}{section}
\newcommand{\prref}[1]{Proposition \ref{#1}}

\def\tp{\otimes}                               

\def\<{\langle}
\def\>{\rangle}

\def\d{\partial}

\def\st{\; | \;}                               

\def\mmod{\;\mathrm{mod}\;}

\def\di{{\mathrm{d}}}

\newcommand{\kk}{\mathbf{k}}       
\newcommand{\NN}{\mathbb{N}}       
\newcommand{\Nset}{\mathbb{N}}       
\newcommand{\ZZ}{\mathbb{Z}}       

\newcommand{\fa}{\mathfrak{a}}
\newcommand{\fe}{\mathfrak{e}}
\newcommand{\fg}{\mathfrak{g}}

\newcommand{\m}{\mathfrak{m}}

\def\de{\delta}
\def\De{\Delta}
\def\ep{\varepsilon}

\def\om{\omega}

\def\th{\theta}

\def\g{{\mathfrak{g}}}      
\def\h{{\mathfrak{h}}}      
\def\dd{{\mathfrak{d}}}

\def\t{{\mathfrak{t}}}

\def\p{{\mathfrak{p}}}
\def\gl{{\mathfrak{gl}}}
\def\gld{\gl\,\dd}
\def\sl{{\mathfrak{sl}}}
\def\sld{\sl\,\dd}

\def\sp{{\mathfrak{sp}}}
\def\spd{\sp\,\dd}
\def\spdd{\sp(\dd, \omega)}

\def\csp{{\mathfrak{csp}}}

\def\Wd{W(\dd)}    
\def\Sd{S(\dd,\chi)}
\def\Hd{H(\dd,\chi,\om)}
\def\Hdzero{H(\dd,0,\om)}
\def\Kd{K(\dd,\th)}

\def\L{{\mathcal{L}}}
\def\I{{\mathcal{I}}}
\def\V{{\mathcal{V}}}
\def\VW{{\mathcal{V}^{\mathrm{W}}}}
\def\VS{{\mathcal{V}_\chi^{\mathrm{S}}}}
\def\VH{{\mathcal{V}_{\chi, \omega}^{\mathrm{H}}}}
\def\VK{{\mathcal{V}_\theta^{\mathrm{K}}}}
\def\VHt{{\mathcal{V}_{\chi, \omega, t, \dd'}^{\mathrm{H}}}}
\def\VHtprime{{\mathcal{V}_{\chi, \omega, t', \dd'}^{\mathrm{H}}}}
\def\VHtR{{\VHt(R)}}
\def\VHtPU{{\VHt(\Pi_+, U)}}
\def\W{{\mathcal{W}}}
\def\S{{\mathcal{S}}}
\def\H{{\mathcal{H}}}
\def\P{{\mathcal{P}}}
\def\K{{\mathcal{K}}}

\def\N{\mathcal{N}}
\def\Z{\mathcal{Z}}
\def\E{\mathcal{E}}

\def\O{\mathcal{O}}

\def\V{\mathcal{V}}           %

\def\ue{U}                 

\DeclareMathOperator{\Span}{span}

\DeclareMathOperator{\Ind}{Ind}
\DeclareMathOperator{\ad}{ad}

\DeclareMathOperator{\sd}{\ltimes}
\DeclareMathOperator{\id}{id}

\DeclareMathOperator{\fil}{F}      

\DeclareMathOperator{\Der}{Der}

\DeclareMathOperator{\Hom}{Hom}

\DeclareMathOperator{\Cur}{Cur}

\DeclareMathOperator{\sing}{sing}

\DeclareMathOperator{\im}{Im}

\begin{document}

\title[Irreducible Modules over Finite Simple Lie Pseudoalgebras IV]
{Irreducible Modules over Finite Simple Lie\\ Pseudoalgebras IV.
Non-primitive Pseudoalgebras}

\author[A.~D'Andrea]{Alessandro D'Andrea}
\address{Dipartimento di Matematica,
Istituto ``Guido Castelnuovo'',
Universit\`a di Roma ``La Sapienza'',
P.le Aldo Moro, 2,
00185 Rome, Italy}
\email{dandrea@mat.uniroma1.it}

\date\today

\begin{abstract}
Let $\dd \subset \dd'$ be finite-dimensional Lie algebras, $H = \ue(\dd), H' = \ue(\dd')$ the corresponding universal enveloping algebras endowed with the canonical commutative Hopf algebra structure. We show that if $L$ is a primitive Lie pseudoalgebra over $H$ then all finite irreducible $L' = \Cur_H^{H'} L$-modules are of the form $\Cur_H^{H'}V$, where $V$ is an irreducible $L$-module, with a single class of exceptions. Indeed, when $L\simeq \Hd$, we introduce non current $L'$-modules $\VHtR$ that are obtained by modifying the current pseudoaction with an extra term depending on an element $t \in \dd' \setminus \dd$, which must satisfy some technical conditions. This, along with results from \cite{BDK1, BDK2, BDK3}, completes the classification of finite irreducible modules of finite simple Lie pseudoalgebras over the universal enveloping algebra of a finite-dimensional Lie algebra.
\end{abstract}

\maketitle


\section{Introduction}

This is the last issue in a series of papers 
addressing the structure and representations of simple Lie pseudoalgebras over a cocommutative Hopf algebra $H = \ue(\dd)$, where $\dd$ is a finite-dimensional Lie algebra. A classification of finite irreducible modules over all primitive simple Lie pseudoalgebras \cite{BDK} has already been achieved in \cite{BDK1, BDK2, BDK3}; in this paper, I give a complete description of finite irreducible representations of non-primitive ones.

The Lie pseudoalgebra language provides a common generalization to both Lie algebras and Lie conformal algebras \cite {DK}, that are strictly related to algebraic properties of the Operator Product Expansion in vertex algebras \cite{K}.
Finite simple Lie pseudoalgebras over $H = \ue(\dd)$ have been classified in \cite{BDK}: they all arise from the so-called {\em primitive} simple Lie pseudoalgebra by means of a current construction. The list of primitive Lie $H$-pseudoalgebras, up to isomorphism, is given in Section \ref{sprimitive}. Apart from simple finite-dimensional Lie algebras, which occur when $\dd = (0)$, they all arise as subalgebras of $\Wd$, see Example \ref{wd}, and are denoted by $\Sd, \Hd, \Kd$. While $\Wd$ and $\Sd$ exist for all choices of the Lie algebra $\dd$, and non isomorphic examples of $\Sd$ are parametrized by $1$-cocycles $\chi \in \dd^*$, instances of $\Hd$, $\Kd$ depend on more elusive properties that $\dd$ must satisfy, see \cite{BDK}.

Representation theory of finite-dimensional simple Lie algebras is certainly well known. Finite irreducible representations of the Lie pseudoalgebras $\Wd$ and $\Sd$, $\Hd$, $\Kd$ have been classified in \cite{BDK1, BDK2, BDK3} respectively: the main result is that each irreducible $L$-module, where $L$ is one of the above Lie pseudoalgebras, arises as a quotient of special representations called {\em tensor modules}. These are parametrized by finite-dimensional irreducible representation of a finite-dimensional Lie algebra associated with $L$; only finitely many of the above tensor modules fail irreducibility, and then fit nicely in complexes that provide pseudoalgebraic translations of differential-geometric constructions such as the de Rham complex, along with its generalizations by Rumin \cite{Ru} and Eastwood \cite{E} in the context of contact and conformally symplectic geometry. The above results are closely connected with the study of finite height representations of Cartan type Lie algebras undertaken by Rudakov \cite{R1, R2} and Kostrikhin \cite{Ko}.

It is not difficult to show, as in Corollary \ref{currirrisirr}, that if $V$ is an irreducible representation of any Lie pseudoalgebra $L$, then the current construction yields an irreducible representation $\Cur_H^{H'}V$ of the current Lie pseudoalgebra $\Cur_H^{H'}L$. The converse does not hold in general, as one may verify by choosing $L$ to be an abelian Lie pseudoalgebra $L$. However, when $L$ is simple, the only case that has been investigated, apart from the trivial one $H = H'$, is $H = \kk, H' = \kk[\d]$ from \cite{CK}, where it is shown that all finite irreducible representations of the current Lie conformal algebra $\Cur \g = \Cur_\kk^{\kk[\d]} \g$ arise as currents of irreducible $\g$-modules. It is thus tempting to conjecture that finite irreducible modules of $\Cur_H^{H'}L$ are always current representations whenever $L$ is a simple primitive Lie pseudoalgebra $L$.

The present paper shows that this expectation holds, with a single class of exceptions: indeed, the simple Lie pseudoalgebra $\Cur_{H}^{H'} \Hd$ may have noncurrent irreducible representations, that are thoroughly classified. The main result is the following\\

\noindent{\bf Theorem
.}
{\em Let $\dd \subset \dd'$ be finite-dimensional Lie algebras, $H \subset H'$ their universal enveloping algebras endowed with the canonical cocommutative Hopf algebra structure. The following is a complete list of finite irreducible representations of the current Lie pseudoalgebra $L' = \Cur_H^{H'} L$, where $L$ is a primitive Lie pseudoalgebra:
\begin{itemize}
\item[---] $\Cur_H^{H'} V$, where $V$ is a finite irreducible $L$-module;
\item[---] $\VHtR$, where $L = \Hd$,  $R$ is a finite-dimensional irreducible representation of $\dd_+ \oplus \spdd$, and $t\in \dd' \setminus \dd$ satisfies 
\begin{itemize}
\item[{\em (i)}] $\ad_\chi t$ preserves $\dd$ and lies in $\spdd$;
\item[{\em (ii)}] $[s, t] = 0,$ where $s$ satisfies $\chi = \iota_s \omega$.
\end{itemize}
\end{itemize}}
\vspace{.3cm}
A description of all nontrivial isomorphism between the above irreducible modules is also provided.
Notice that here $\dd_+$ denotes the extension
$$ 0 \to \kk_\chi \to \dd_+ \to \dd \to 0$$
of $\dd$ by the one-dimensional abelian ideal of $1$-cocycle $\chi$ corresponding to the $2$-cocycle $\omega$ and $\ad_\chi t: \dd \to \dd'$ is defined as $(\ad_\chi t)(\d) := [t, \d] + \chi(\d) t$. Noncurrent representations $\VHtR$ are introduced in Section \ref{shtype}. They are obtained by adding an extra term $(t \otimes 1) \otimes_{H'} (1 \otimes u)$ to the expression providing the pseudoaction $e * (1 \otimes u)$ in a current tensor module $\Cur_H^{H'} \VH(R)$. 

The special behaviour of $\Hd$ depends on the presence of nontrivial central elements in the corresponding annihilation Lie algebra. This fact also plays a major role in the less standard description from \cite{BDK3} of finite irreducible representations of $\Hd$ when compared to the more unified treatment of other primitive types.

Unfortunately, many conflicting notations from \cite{BDK1, BDK2, BDK3} have to be resolved in this paper. Explanations are provided in footnotes whenever needed.

\section{Hopf algebra preliminaries}

\subsection{Hopf notation}

In this paper all vector spaces, algebras and tensor products are, unless otherwise specified, over an algebraically closed field $\kk$ of characteristic zero. We will deal with pairs $\dd \subset \dd'$ of finite dimensional Lie algebras and denote by $H$, respectively $H'$, the universal enveloping algebra $\ue(\dd)$, resp. $\ue(\dd')$. 
{No other Hopf algebras will be considered} with the exception of such universal enveloping algebras.

Notice that both $H$ and $H'$ are  Hopf algebras with respect to the coproduct 
$\De$, antipode $S$, and counit $\ep$ given by:
\begin{equation}\label{des}
\De(\d) = \d\tp1 +1\tp\d \,, \quad S(\d)=-\d \,, \quad \ep(\d)=0 \,,
\qquad \d\in\dd' \,.
\end{equation}
More precisely, $H\subset H'$ is a Hopf subalgebra, so that the inclusion $\iota: H \to H'$ is a Hopf-algebra homomorphism.
We will employ the notation:
\begin{align}
\label{de1}
\De(h) &= h_{(1)} \tp h_{(2)} = h_{(2)} \tp h_{(1)} \,, 
\\
\label{de2}
(\De\tp\id)\De(h) &= (\id\tp\De)\De(h) = h_{(1)} \tp h_{(2)} \tp h_{(3)} \,,
\\
\label{de3}
(S\tp\id)\De(h) &= h_{(-1)} \tp h_{(2)} \,, \qquad\quad h\in H \,.
\end{align}
Then the antipode and counit axioms can be written
as follows:
\begin{align}
\label{antip}
h_{(-1)} h_{(2)} &= h_{(1)} h_{(-2)} = \ep(h),
\\
\label{cou}
\ep(h_{(1)}) h_{(2)} &= h_{(1)} \ep(h_{(2)}) = h,
\end{align}
while the fact that $\De$ is a homomorphism of algebras
translates as:
\begin{equation}
\label{deprod}
(fg)_{(1)} \tp (fg)_{(2)} = f_{(1)} g_{(1)} \tp f_{(2)} g_{(2)},
\qquad f,g\in H.
\end{equation}
Eqs.\ \eqref{antip}, \eqref{cou} imply the following
useful relations:
\begin{equation}
\label{cou2}
h_{(-1)} h_{(2)} \tp h_{(3)} = 1\tp h
= h_{(1)} h_{(-2)} \tp h_{(3)}.
\end{equation}
%

%

Set $\dim \dd = N,\, \dim \dd' = N + r$. If $\{\d_1, \dots,
\d_{N+r}\}$ is a basis of $\dd'$, we denote by $\{x^1,\dots,x^{N+r}\}$ the corresponding dual basis of $\dd'^*$.
The structure constant $c_{ij}^k$ of $\dd'$, are defined by 
\begin{equation}\label{cijk}
[\d_i,\d_j]=\sum_{k=1}^{N+r} c_{ij}^k\d_k
\,, \qquad i,j=1,\dots,N+r \,.
\end{equation}
We have a corresponding (reduced) Poincaré-Birkhoff-Witt basis
\begin{equation}\label{dpbw}
\d^{(K)} = \d_1^{k_1} \dotsm \d_{N+r}^{k_{N+r}} / k_1! \dotsm
k_{N+r}! \,, \qquad K = (k_1,\dots,k_{N+r}) \in\ZZ_+^{N+r} \,
\end{equation}
of $H'$.
\begin{remark}
PBW bases may be used to show that $H'$ is free both as a left and as a right $H$-module. The embedding $\iota: H \hookrightarrow H'$ is thus a pure homomorphism.
\end{remark}

\subsection{Straightening}\label{straight}

Let $M$ be a left $H$-module. The coproduct $\Delta: H \to H \otimes H$ makes $H \otimes H$ into a right $H$-module, so that it makes sense to consider the tensor product $(H \otimes H) \otimes_H M$, which is a left $H \otimes H$-module by left multiplication on the first $\otimes_H$ factor. Elements lying in $(H \otimes H) \otimes_H M$ may be expressed in very many ways, and this can make it difficult to verify whether two distinct expressions actually describe the same quantity. This problem is solved by the left- and right-straightening techniques, introduced in \cite{BDK}, that we recall here briefly.

\begin{proposition}\label{straightprop}
Let $M$, resp. $N$, be a right, resp. left, $H$-module. Then the assignments
\begin{align}
M \otimes N \ni m \otimes n & \mapsto (m \otimes 1) \otimes_H n \in (M \otimes H) \otimes_H N\label{left2}\\
M \otimes N \ni m\otimes n & \mapsto (1 \otimes m) \otimes_H n \in (H \otimes M) \otimes_H N\label{right2}
\end{align}
extend to linear isomorphisms. 
\end{proposition}
\begin{proof}
One checks easily that the maps
\begin{align}
(M \otimes H) \otimes_H N \ni (m \otimes h) \otimes_H n & \mapsto mh_{(-1)} \otimes h_{(2)}n\\
(H \otimes M) \otimes_H N \ni (h \otimes m) \otimes_H n & \mapsto mh_{(-2)} \otimes h_{(1)}n
\end{align}
extend to explicit inverses to \eqref{left2}, \eqref{right2}.
\end{proof}

\begin{corollary}\label{pretechnical}
Let $N$ be a left $H$-module. 
Every element $\alpha \in (H \otimes H) \otimes_H N$ can be expressed either in the form
\begin{equation}\label{left}
\alpha = \sum_i (h_i \otimes 1) \otimes_H m^i, \quad h_i \in H, m^i \in N
\end{equation}
or in the form
\begin{equation}\label{right}
\alpha = \sum_i (1 \otimes k_i) \otimes_H n^i, \quad k_i \in H, n^i \in N.
\end{equation}
Similarly, every element in $(H \otimes H \otimes H) \otimes_H N$ has a unique representative both in $(H \otimes H \otimes \kk) \otimes_H N$ and in $(\kk \otimes H \otimes H) \otimes_H N$.
\end{corollary}
\begin{proof}
Use $M = H$, resp. $H \otimes H$, in Proposition \ref{straightprop}.
\end{proof}

\begin{corollary}\label{technical}
Let $N \subset M$ be left $H$-modules, and
$$\alpha = \sum_{i} (h_i \otimes 1) \otimes_H m^i \in (H \otimes H) \otimes_H M = \sum_i (1 \otimes k_i) \otimes_H n^i,$$
where $h_i$, resp. $k_i$, are linearly independent elements in $H$. Then $\alpha$ lies in $(H \otimes H) \otimes_H N$ if and only if $m^i \in N$, or equivalently $n^i \in N$, for every $i$. 
\end{corollary}

\begin{remark}\label{currentdoesnothing}
The above corollary shows that there always exists a smallest $N \subset M$ such that $\alpha \in (H \otimes H) \otimes_H N$, which may be computed by taking the left- or right-straightened expression for $\alpha$, and considering the $H$-linear span of all elements on the right of $\otimes_H$.

Corollary \ref{technical} also allows one to check equalities in $(H \otimes H) \otimes_H N$ and $(H \otimes H \otimes H) \otimes_H N$. It is enough to straighten everything, say to the left, and then verify the equality in the vector spaces $H \otimes N$, $H \otimes H \otimes N$ respectively, which is a trivial job.
\end{remark}

We end this section with a technical statement which can be thought of as a partial straightening.
We split the set $\NN^{n}$ of PBW indices for $H = \ue(\dd)$ as follows: elements of $R: = \NN^k \times \{0\}\subset \NN^{n}$ are supported on the first $k$ indices, whereas those lying in $S: = \{0\} \times \NN^{n-k}$ vanish on the first $k$ indices, so that they are supported on the subsequent $n-k$ ones. Clearly, if $K_i \in R, L_i \in S, i = 1, 2,$ then $K_1 + L_1 = K_2 + L_2$ if and only if $K_1 = K_2, L_1 = L_2$. Moreover, the sum of two elements in $\NN^{n}$ lies in $R$ (resp. $S$) if and only if both summands lie in $R$ (resp. $S$).

\begin{lemma}\label{RS}
Let $U \subset V$ be left $H$-modules. Then the element
\begin{equation}\label{alpha}
\alpha = \sum_{K \in R, L \in S} (\d^{(K)} \otimes \d^{(L)}) \otimes_{H} v_{K+L} \in (H \otimes H) \otimes_{H} V
\end{equation}
lies in $(H \otimes H) \otimes_{H} U$ precisely when all $v_{J}, J \in \NN^n,$ lie in $U$.
\end{lemma}
\begin{proof}
Clearly, if all $v_{K, L}$ lie in $U$, then $\alpha \in (H \otimes H) \otimes_{H} U$. As for the converse, assume $\alpha \in (H \otimes H) \otimes_{H} U$ and left-straighten \eqref{alpha} to obtain
\begin{align}
\alpha & = \sum_{K \in R}\sum_{L_1, L_2 \in S} (\d^{(K)}S(\d^{(L_1)}) \otimes 1) \otimes_{H} \d^{(L_2)} v_{K+L_1 + L_2} \\
& = \sum_{K \in R}\sum_{L_1, L_2 \in S} (\d^{(K+L_1)}) \otimes 1) \otimes_{H} (-1)^{|L_1|}\d^{(L_2)} v_{K+L_1 + L_2}\\
& = \sum_{J \in \ZZ^{n}} (\d^{(J)} \otimes 1) \otimes_{H} \sum_{L_2 \in S} \pm\, \d^{(L_2)} v_{J + L_2} \in (H \otimes H) \otimes_{H} U,
\end{align}
where we set $J = K+L_1$. Then Corollary \ref{technical} shows that
\begin{equation}\label{strangesum}
\sum_{L_2 \in S} \pm\, \d^{(L_2)} v_{J + L_2}
\end{equation}
lies in $U$ for each choice of $J \in \ZZ^{n}, L \in S$.

Now, proceed by contradiction and assume that $J\in \ZZ^{n}$ is (lexicographically)  maximal such that $v_{J} \notin U$. Then \eqref{strangesum} rewrites as the sum
$$\pm v_{J} + \sum_{0 \neq L_2 \in S} \pm \d^{(L_2)} v_{J + L_2} \in U.$$
However, the summation on the right lies in the $H$-submodule $U\subset V$ by maximality of $J$, whence $v_{J} \in U$, yielding a contradiction.
\end{proof}

\begin{remark}\label{RSremark}
It is important to stress here that the above Lemma \ref{RS} may be applied to any totally ordered basis of a Lie algebra. In particular, it may be applied to any basis of $\dd'$, and not just those where the last elements constitute a basis of $\dd\subset \dd'$ as we will consider below; even in this case, $k$ does not need to equal the difference $\dim \dd' - \dim \dd$ but may be any number.
\end{remark}

\begin{corollary}\label{RSuniqueness}
Let $V$ be a left $H$-module. Then equality
$$\sum_{K \in R, L \in S} (\d^{(K)} \otimes \d^{(L)}) \otimes_{H} u_{K+L} = \sum_{K \in R, L \in S} (\d^{(K)} \otimes \d^{(L)}) \otimes_{H} v_{K+L}$$
holds in $(H \otimes H) \otimes_H V$ if and only if $u_N = v_N$ for every $N \in \NN^{n}$.
\end{corollary}
\begin{proof}
Use $U = 0$ in Lemma \ref{RS}.
\end{proof}

\subsection{Filtrations}

The PBW basis may be used to set up a canonical increasing filtration of $H'$ given by
\begin{equation}\label{filued1}
\fil^p H' = \Span_\kk\{ \d^{(K)} \st |K| \le p \} \,, \qquad
\text{where} \quad |K|=k_1+\cdots+k_{N+r} \,.
\end{equation}
This filtration does not depend on the choice of basis of $\dd'$, and 
is compatible with the Hopf algebra structure of $H$
. We have: $\fil^{-1} H' =
\{0\}$, $\fil^0 H' = \kk$, and $\fil^1 H' = \kk\oplus\dd'$.

The dual $X':= \Hom_\kk(H',\kk)$ is a commutative associative algebra and inherits an induced decreasing filtration $\fil_i X' = (\fil^i H')^\perp$ making it into a linearly compact vector space.
We will identify $\dd'^*$ as a subspace of $X'$ by letting 
$\langle x^i, \d_i \rangle = 1$ and $\langle x^i, \d^{(I)} \rangle = 0$
for all other basis vectors \eqref{dpbw}.
Mapping $x^i \mapsto t^i$ gives rise to an isomorphism from $X$ 
to the algebra $\kk[[t^1,\dots,t^{N+r}]]$ 
of formal power series in $N+r$ indeterminates. Notice that the above filtration satisfies $\fil_{-1}X' = X'$, whereas $\fil_{0}X'$ is the (maximal) ideal generated by $x^1, \dots, x^{N+r}$. Similarly, $\fil_{p}X'$ coincides with $(\fil_{0} X')^{p+1}$.

\begin{remark}\label{donotadd}
The choice of indices in the filtration for $X'$ is natural, but somewhat clumsy, and one has $(\fil_{p} X')(\fil_{q} X') = \fil_{p+q+1} X'$ so that indices do not add up properly. This will later have consequences in choosing the right indexing of the filtration of annihilation algebras.
\end{remark}

\begin{remark}
If $H = \ue(\dd)$, then $H \otimes H = \ue(\dd) \otimes \ue(\dd) \simeq \ue(\dd \oplus \dd)$. We will occasionally denote by $\fil^p (H \otimes H)$ the corresponding filtration, which satisfies $\fil^p (H \otimes H) = \sum_{i+j = p} \fil^i H \otimes \fil^j H$.
\end{remark}

The Lie algebra $\dd'$ has left and right actions on $X'$ by derivations, given by
\begin{align}
\label{dx1}
\langle \d x, h\rangle &= -\langle x, \d h\rangle \,,
\\
\label{dx2}
\langle x \d, h\rangle &= -\langle x, h\d\rangle \,,
\qquad \d\in\dd' \,, \; x \in X' \,, \; h \in H' \,,
\end{align}
where $\d h$ and $h\d$ are the products in $H'$. These two actions
coincide only when $\dd'$ is abelian, the difference
$\d x - x \d$ giving the coadjoint action of $\d\in\dd'$ on $x\in X'$:
$$\langle \d x - x \d, h\rangle = -\langle x, [\d, h]\rangle.$$

\subsection{Splitting a projection}

Throughout the rest of the paper, we will chose a basis  $\{\d_1, \dots,
\d_{N+r}\}$ of $\dd'$ so that the last $N$ elements $\d_{r+1}, \dots, \d_{N+r}$ are a basis of $\dd\subset \dd'$. Then the PBW basis of $H$ gets identified to a subset of the PBW basis of $H'$; correspondingly, the subalgebra $\kk[[x^{r+1}, \dots, x^{N+r}]] \subset X'$ may be identified with $X$. Notice that the inclusion $\iota: H \to H'$ induces a surjective commutative algebra homomorphism $\iota^*: X' \to X$ whose kernel we denote by $I = \ker \iota^*$, and the above identification provides a splitting $\varsigma: X \to X'$ to $\iota^*$, so that $X' = X + I$ is a direct sum decomposition (as vector spaces). More specifically, $\iota^*$ is the unique homomorphism mapping $x^1, \dots, x^r$ to $0$ and fixing each $x^{r+1}, \dots, x^{N+r}$; thus, $I\subset X'$ is the ideal generated by $x^1, \dots, x^r$.

\begin{remark}
Denote by $H_+$ the augmentation ideal of $H$. It is easy to show that the subalgebra $\kk[[x^1, \dots, x^r]]\subset X'$ equals $(H' H_+)^\perp$ and is thus canonically determined, and independent of the choice of the basis of $\dd'$. \end{remark}
Henceforth, we will denote this canonical subalgebra of $X'$ by $\O$, and its unique maximal ideal by $\m=\langle x^1, \dots, x^r\rangle$. Notice that one has an isomorphism $X' \simeq \O \,\widehat{\otimes}\, X$ of linearly compact vector spaces, which identifies $I$ with $\m \,\widehat{\otimes}\, X$.\\

It is important to highlight the fact that we are provided with two distinct pairs of left and right actions of $\dd$ on $X$: one is as above in the case where $r = 0$; the other is obtained by restricting to the subalgebra $\dd$ the action of $\dd'$ on $X \simeq \varsigma(X) \subset X'$. In general, these actions will differ, but they are still nicely compatible. 

\begin{proposition}\label{dactonxprime}
\begin{enumerate}
\item[{\em (i)}] The right action of $\dd \subset \dd'$ on $X'$ preserves $X \subset X'$ and coincides with the natural right action of $\dd$ on $X$. Moreover, the right action of $\dd$ on $x^1, \dots, x^r$ is trivial.
\item[{\em (ii)}] The right action of $\dd$ on $X'$ is $\O$-linear.
\item[{\em (iii)}] The restriction to $X \subset X'$ of the left action of $\dd \subset \dd'$ on $X'$ coincides, up to elements in $I$, with the natural left action of $\dd$ on $X$. 
\item[{\em (iv)}] $I$ is stabilized by both the left- and right- action of $\dd \subset \dd'$ on $X'$.
\end{enumerate}
\end{proposition}
\begin{proof}
\mbox{ }
\begin{enumerate}
\item[{\em (i)}] 
Let $\d \in \dd$. Then by definition $\langle x^k\d, h\rangle = - \langle x^k, h\partial\rangle$. When $h \in \ue(\dd) \subset \ue(\dd')$, then $h\d$ lies in the augmentation ideal of $\ue(\dd)$; consequently, when $h$ is a member of the PBW basis \eqref{dpbw} of $\ue(\dd')$, the only contribution to $\langle x^k, \d^{(K)}\cdot \d\rangle,$ where $r+1 \leq k \leq N+r$, will arise from those $K$ that are supported in indices $r+1$ through $N+r$.
Thus, the right action of $\d\in \dd$, viewed as an element of $\dd'$, on $x^k\in X$, viewed as an element of $X'$, will coincide with the right action of $\d$ on $x^k$. The first claim now follows by observing that the right action od $\dd'$ on $X'$ is by (continuous) derivations, and that $x^k, r+1\leq k \leq N+r,$ are (topological) algebra generators of $X \subset X'$, so that the right action of $\dd\subset \dd'$ on $X\subset X'$ is uniquely determined by its action on elements $x^k, r+1 \leq k \leq N+r$.

As for the second claim, we similarly argue that if $k\leq r$ and $\d \in \dd$, then $\langle x^k\d, \d^{(K)}\rangle = - \langle x^k, \d^{(K)}\d\rangle$ vanishes, as $\d^{(K)}\d$ lies in the left ideal of $\ue(\dd')$ generated by the augmentation ideal of $\ue(\dd) \subset \ue(\dd')$, which lies in $(x^k)^\perp$.

\item[{\em (ii)}] Follows immediately from part {\em (i)} and continuity of the action of $\dd'$ on $X'$.

\item[{\em (iii)}]
It is enough to compute $\langle \d x^k, \d^{(K)}\rangle = - \langle x^k, \d \d^{(K)}\rangle$ when $r+1 \leq k \leq N+r$ and $K$ is supported over the same indices, as other choices of $K$ will yield contributions lying in the ideal $I$.
However, only the structure of $\dd$ is involved in the computation of the above expressions $\langle x^k, \d \d^{(K)}\rangle$.


\item[{\em (iv)}] The ideal $I = \ker \iota^*\subset X'$ coincides with $H^\perp$. The claim follows from the fact that both left- and right- multiplication by elements of $\dd$ stabilize $H \subset H'$.
\end{enumerate}
\end{proof}

\section{Lie pseudoalgebra preliminaries}

\subsection{Pseudoalgebraic definitions}

Let $H = \ue(\dd)$. An $H$-pseudoalgebra is a left $H$-module $L$ endowed with an $H\otimes H$-linear {\em pseudoproduct} $L \otimes L \to (H \otimes H) \otimes_H L$, where $(H \otimes H) \otimes_H L$ is defined as in Section \ref{straight}.

A pseudoproduct is usually denoted by $a \otimes b \mapsto a*b$, and one may make sense of $(a*b)*c, a*(b*c)$ as elements in $(H \otimes H \otimes H) \otimes_H L$, as in \cite[(3.15)-(3.19)]{BDK}. Then a Lie pseudoalgebra is a pseudoalgebra whose pseudoproduct satisfies a {\em pseudo-version} of the skew-symmetry and Jacobi identity axioms for a Lie algebra. In this context, the pseudoproduct is called {\em Lie pseudobracket} and the most usual notation for it is $[a * b]$. The correct pseudoalgebraic translation of the axioms is
\begin{align}
[a * b] & = - (\sigma \otimes_H \id_L) [b*a];\\
[a * [b * c]] & = [[a * b] * c] + ((\sigma \otimes \id_H) \otimes_H \id_L) [b*[a*c]],
\end{align}
where $a, b, c \in L$ and $\sigma: H \otimes H \to H \otimes H$ is the {\em flip} map switching the two tensor factors. If $A, B \subset L$ are $H$-submodules, then it is convenient to define $[A,B]$ as the smallest $H$-submodule $S \subset L$ such that $[a*b] \in (H \otimes H) \otimes_H S$ for all $a\in A, b \in B$. Then $A \subset L$ is a {\em subalgebra} if $[A,A] \subset A$, and an ideal if $[L,I] \subset I$. A Lie pseudoalgebra $L$ is {\em abelian} if $[L,L] = 0$ and {\em simple} if it is not abelian and its only ideals are the trivial ones $(0), L$.

One may similarly define representations of pseudoalgebras. In our setting, if $L$ is an $H$-Lie pseudoalgebra and $M$ a left $H$-module, we may consider a {\em pseudoaction} to be an $H \otimes H$-linear map $L \otimes M \to (H \otimes H) \otimes_H M$ that we will denote by $a\otimes m \mapsto a*m$. This defines a {\em Lie pseudoalgebra representation} when
\begin{equation}\label{lieaction}
[a * b] * m = a * (b * m) - ((\sigma \otimes \id_H) \otimes_H \id_M)\,\, b * (a * m),
\end{equation}
for all $a, b \in L, m \in M,$ and equality is understood to hold inside $(H \otimes H \otimes H) \otimes_H M$ as before. We will also say that $M$ is an $L$-module.

\begin{remark}
We stress the fact that a pseudoaction of a Lie pseudoalgebra $L$ on the left $H$-module $M$ is nothing else than a $H \otimes H$-linear maps $L \otimes M \to (H \otimes H) \otimes_H M$. This only makes $M$ into a Lie pseudoalgebra representation of $L$ when \eqref{lieaction} is satisfied.
\end{remark}

Once again, if $A \subset L, N \subset M$, one may define $A\cdot N$ to be the smallest $H$-submodule $S\subset M$ such that $[a*n] \in (H \otimes H) \otimes_H S$ for all $a \in L, n \in N$. Then $N\subset M$ is an $L$-submodule if $L\cdot N \subset N$ and  $M$ {\em has a trivial action} of $L$ if $L\cdot M = (0)$; notice that when the action of $L$ on $M$ is trivial, then any $H$-submodule of $M$ is automatically an $L$-submodule.

An $L$-module $M$ is {\em irreducible} if its only $L$-submodules are $(0), M$ and {\bf it does not have a trivial action} of $L$. An $L$-submodule $N \subsetneq M$ is {\em maximal} if the only submodules of $M$ containing it are $N$ and $M$; then $M/N$ is either an irreducible $L$-module or it is a simple (nonzero) $H$-module with a trivial action of $L$. In particular, if $M/N$ is an irreducible $L$-module, then $L\cdot M = M$.

\begin{remark}
When $A, B \subset L$, then $[A, B]$ denotes an $H$-submodules of $L$, whereas we reserve the notation $[A * B]$ for the subset of $(H \otimes H) \otimes_H L$ containing all $[a * b], a \in A, b \in B$. The same applies to $A \cdot B, A * B$, where $A \subset L, B \subset M$ and $M$ is a Lie pseudoalgebra representation of $L$.
\end{remark}

\subsection{Annihilation algebras}\label{annalg}

Let $L$ be a Lie pseudoalgebra over $H$, and denote as usual by $X = H^*$ the commutative algebra dual to $H$. Then $\L = X \otimes_H L$ may be endowed with a bilinear product defined by
\begin{equation}\label{annihbracket}
[x \otimes_H a, y \otimes_H b] = \sum_i (x h^i)(y k^i) \otimes_H c_i,
\end{equation}
as soon as the Lie pseudobracket on $L$ satisfies $[a * b] = \sum_i (h^i \otimes k^i) \otimes_H c_i$. Then $H$-bilinearity along with the pseudoalgebra analogue of skew-symmetry and Jacobi identity for $[\,\,*\,\,]$ ensure that \eqref{annihbracket} is a Lie bracket on $\L$.

Let now $\dd$ be a finite-dimensional Lie algebra, $H = \ue(\dd)$ the corresponding universal enveloping algebra, and assume that the {\em finitely generated} $H$-module $L$ admits a Lie pseudoalgebra structure over $H$. Then we may use the filtration on $X = H^*$ so as to build up a corresponding linearly compact topology on $\L$.

More explicitly, choose a finite dimensional vector subspace $S \subset L$ such that $L = HS$. If $\{s_i\}$ is a basis of $S$, then
$$[s_i * s_j] = \sum_{i} (h_{ij}^k \otimes k_{ij}^k) \otimes_H s_k$$
for some choice of $h_{ij}^k, k_{ij}^k \in H$, and one may find $\ell \in \Nset$ so that $h_{ij}^k \otimes k_{ij}^k \in \fil^{\ell} (H \otimes H)$ for all choices of $i, j, k$.

Setting $\L_i = (\fil_{i+\ell-1} X) \otimes_H S$ then provides $\L$ with a decreasing filtration 
\begin{equation}\label{filtra}
\L = \L_{-\ell} \supset \L_{-\ell+1} \supset \dots \supset \L_{0} \supset \L_1 \supset \dots
\end{equation}
which makes it into a linearly compact vector space, and satisfies $[\L_i , \L_j] \subset \L_{i+j}$ for all $i, j$. Thus, the Lie bracket is continuous with respect to the topology and $\L$ is a linearly compact topological Lie algebra. Notice that the filtration depends on the choice of the generating subspace $S$, but the topology it induces does not. Some choices of $S$ allow for more convenient, i.e., lower, values of $\ell$. In next section, we exhibit some convenient choices of $S$ when $L$ is a primitive Lie pseudoalgebra.

\begin{proposition}
The action of $H$ on $X$ is continuous. In particular, both the right- and the left-action of $\dd$ on $X$ is by continuous derivations.
\end{proposition}
\begin{proof}
It follows from $(\fil^p H).(\fil_n \L) \subset \fil_{n-p} \L$ and the fact that elements in $\dd\subset H$ are primitive in $H$. The right-action case is done in the same way.
\end{proof}

When dealing with representations of the Lie pseudoalgebra $L$, it is also convenient to introduce the so-called {\em extended annihilation algebra} $\widetilde \L$. This is the semi-direct product $\widetilde \L:= \dd \sd \L$, where the adjoint action of $\dd$ on $\L$ is defined as
$$[\d, x \otimes_H a]:= (\d x) \otimes_H a.$$

\begin{example}\label{elie}
As a trivial example, let us consider the case $\dd = (0)$, so that $H = \kk$.
If $L$ is a finite $H$-Lie pseudoalgebra, then $(H \otimes H) \otimes_H L = (\kk \otimes \kk) \otimes_\kk L$ can be canonically identified with $L$. Thus the Lie pseudobracket is simply a (bilinear) Lie bracket, and $L$ is a finite-dimensional Lie algebra.

The corresponding annihilation algebra $\L = X \otimes_H L = (\kk)^* \otimes_\kk L$ is canonically isomorphic to the Lie algebra $L$. The filtration on both $H$ and $X$ is trivial, and the corresponding topology on $\L$ is discrete; however, the concepts of discrete and linearly compact topologies coincide for finite-dimensional vectors spaces. As $\dd = (0)$, we also get $\widetilde \L = \L \simeq L$.
\end{example}

The relevance of the extended annihilation algebra lies in the following fact:
\begin{theorem}[{\cite[Proposition 9.4]{BDK}}]\label{repdescription}
The notions of Lie pseudoalgebra action of $L$ on the $H$-module $V$ is equivalent to that of continuous action of the Lie algebra $\widetilde \L$ on $V$ endowed with the discrete topology. More specifically, if $[a * v] = \sum_i (f^i \otimes g^i) \otimes_H v_i$, then
$$(x \otimes_h a).v = \sum_i \langle S(x), f^i g^i_{(-1)}\rangle g^i_{(2)}v_i.$$

Conversely, if $V$ is a discrete module over $\widetilde L$, one may use the $\dd \subset \widetilde \L$-action on $V$ in order to endow it with an $H = \ue(\dd)$-module structure, and recover the pseudoaction by
$$a*v = \sum_{i} (S(h^i) \otimes 1) \otimes_H (x_i \otimes_H a).v,$$
where $\{h^i\}$ and $\{x_i\}$ are dual bases of $H$ and $X$.
\end{theorem}

It is known that the above equivalence preserves the natural notions of irreducibility, so that the study of irreducible representations of a Lie pseudoalgebra $L$ translates into that of irreducible representations of the corresponding extended annihilation algebra $\widetilde \L$.

We end this section by recalling a fact stated in \cite[Lemma 14.4]{BDK} that we are going to use multiple times. If $V$ is a Lie pseudoalgebra representation of $L$, set
\begin{itemize}\label{kern}
\item $\ker V:=\{v \in V\,|\, L*v = 0\}$;
\item $\ker_n V:= \{v \in V\,|\,\L_n.v = 0\}$,
\end{itemize}
whereas $\L_n$ is defined as in \eqref{filtra}, according to some generating subspace $S \subset L$. Then
\begin{proposition}
If $V$ is a finite $L$-module, then the quotient $\ker_n V/\ker V$ is a finite-dimensional vector space for every choice of $n\geq -\ell$, and it is nonzero when $n$ is sufficiently large.
\end{proposition}

\section{Primitive simple Lie pseudoalgebras}\label{sprimitive}

\subsection{Examples of primitive simple Lie pseudoalgebras}

Let $\dd$ be a finite dimensional Lie algebra.
In this paper, by {\em primitive Lie pseudoalgebra}, we mean a simple Lie pseudoalgebra over the cocommutative Hopf algebra $H = \ue(\dd)$ which cannot be obtained by means of a non-trivial current construction, see \cite[Section 4.2]{BDK} and Section \ref{scurrent} below. More specifically, they are either finite-dimensional simple Lie algebras over $H = \kk = \ue(\{0\})$, or one of the {\em primitive pseudoalgebras of vector fields} from \cite[Section 8]{BDK}. Let us review them briefly.

\begin{example}[Simple Lie algebras]\label{lie}
Let $\g$ be a finite dimensional simple Lie algebra over $H = \kk$. As in Example \ref{elie} its Lie bracket may be rewritten in a pseudoalgebraic fashion as follows:
$$[a * b] = (1 \otimes 1) \otimes_\kk [a, b], \qquad a, b \in \g.$$
Here we like to stress the trivial fact that $S = \g$ is a finite-dimensional vector space with the property that
$$[S * S] \in (1 \otimes 1) \otimes_\kk S.$$
We have already seen that the (extended) annihilation algebra of $\g$ is $\g$ itself.
\end{example}

\begin{example}[$\Wd$]\label{wd}
Let $\dd$ be a (nonzero) finite dimensional Lie algebra over $\kk$, $H = \ue(\dd)$. Then $L = \Wd = H \otimes \dd$ is a simple Lie pseudoalgebra when endowed with the unique pseudo-Lie bracket $H$-bilinearly extending
\begin{equation}\label{wdbracket}
[1 \otimes a * 1 \otimes b] = (1 \otimes 1) \otimes_H (1 \otimes [a, b]) + (b \otimes 1) \otimes_H (1 \otimes a) - (1 \otimes a) \otimes_H (1 \otimes b),
\end{equation}
where $a, b, \in \dd$.
Then $S = \kk \otimes \dd \subset H \otimes \dd$ is a finite dimensional vector subspace of $L$ which generates it as an $H$-module and direct inspection of \eqref{wdbracket} shows that
$$[S * S] \in (\dd\otimes \kk + \kk \otimes \dd + \kk \otimes \kk) \otimes_H S.$$
The annihilation algebra $\W = X \otimes_H \Wd$ is isomorphic, as a Lie algebra, to the Lie algebra $W_N$ from Cartan's classification \cite{Cartan}. We denote by $\E\in \W$ the element corresponding to the Euler vector field $E = \sum_{i=1}^N x^i \, \d/\d x^i\in W_N$ under such isomorphism.
\end{example}

\begin{example}[$\Sd$]
Let $\dd$ be a finite dimensional Lie algebra over $\kk$, $\chi \in \dd^*$ a trace form, and set $\Sd:=\left\{ \sum_i h^i \otimes \d_i \in \Wd\,|\, \sum_i h^i(\d_i + \chi(\d_i)) = 0\right\}.$ Then $\Sd \subset \Wd$ is the $H$-submodule generated by elements
$$s_{ab} = (a + \chi(a)) \otimes b - (b + \chi(b)) \otimes a - 1\otimes [a, b],\qquad a, b \in \dd,$$
and is a simple subalgebra of $\Wd$. Lie pseudobrackets between the above elements may be read off \cite[Proposition 8.4]{BDK}.
If we denote by $S$ the $\kk$-linear span of elements $s_{ab}, a, b \in \dd$, then $S$ is a finite-dimensional vector subspace of $\Sd$ generating it as an $H$-module, and satisfying
$$[S * S] \in ((\dd + \kk) \otimes (\dd + \kk)) \otimes_H S.$$
The annihilation algebra $\S = X \otimes_H \Sd$ is isomorphic, as a Lie algebra, to the Lie algebra $S_N$ from Cartan's list.
\end{example}

\begin{example}[$\Hd$ and $\Kd$]\label{hkd}
Let $\dd$ be a finite dimensional Lie algebra over $\kk$, and let the rank one free $H$-module $L = He$ support a Lie pseudoalgebra structure over $H = \ue(\dd)$. Then one may see \cite[Section 4.3]{BDK} that
\begin{equation}\label{hkbracket}
[e * e] = (r + s\otimes 1 - 1 \otimes s) \otimes_H e
\end{equation}
where $0 \neq r \in \bigwedge^2 \dd, s \in \dd$ satisfy opportune conditions, under which 
\begin{equation}\label{HKembeds}
e \mapsto -r + 1 \otimes s \in H \otimes \dd \simeq \Wd
\end{equation}
extends to an injective homomorphism of Lie pseudoalgebras $L \hookrightarrow \Wd$.

We obtain the primitive Lie pseudoalgebras $\Hd$, respectively $\Kd$, when $\dd$ is even dimensional and $r$ is non-degenerate, resp. when $\dd$ is odd dimensional and  it is linearly generated by $s$, along with the support of $r$.
Once again, choosing $S = \kk e$ provides a finite-dimensional vector subspace of $L$ satisfying 
$$[S * S] \in ((\dd + \kk) \otimes (\dd + \kk)) \otimes_H S.$$

The annihilation Lie algebra $\K = X \otimes_H \Kd$ is isomorphic to the Cartan type Lie algebra $K_N$, and we denote by $\E'\in \K$ the element corresponding to the Euler vector field $E' = 2x^1 \d/\d x^1 + \sum_{i=2}^{N} x^i \, \d/\d x^i\in K_N$ under such isomorphism. The annihilation algebra $\H = X \otimes_H \Hd$ is instead isomorphic to the unique irreducible central extension $P_N$ of the Cartan Lie algebra $H_N$.
\end{example}

We may sum up the above examples in the following
\begin{proposition}\label{lowdegree}
Let $\dd$ be a finite dimensional Lie algebra and $L$ be a simple primitive Lie pseudoalgebra over the cocommutative Hopf algebra $H = \ue(\dd)$ as in Examples \ref{lie}-\ref{hkd}. Then there exists a finite dimensional subspace $S \subset L$ and $\ell \in \NN$ such that $L = HS$ and
$$[S * S] \in ((\dd + \kk) \otimes (\dd + \kk)\cap \fil^\ell(H \otimes H)) \otimes_H S,$$
where
\begin{itemize}
\item[---] $\ell=0$ if $L = \g$ is a finite-dimensional simple Lie algebra;
\item[---] $\ell=1$ if $L = \Wd$;
\item[---] $\ell=2$ if $L = \Sd, \Hd, \Kd$.
\end{itemize}
\end{proposition}

\subsection{Annihilation algebra of primitive Lie pseudoalgebras of vector fields}
\label{nonontrivial}

In this section, we recall a few facts from \cite[Section 6]{BDK} on infinite-dimensional simple linearly compact Lie algebras and their irreducible central extensions. Here $L$ is a primitive simple finite Lie pseudoalgebra over $H = \ue(\dd)$, where $\dd$ is a Lie algebra of finite dimension $N>0$.

We have already seen in Examples \ref{wd}-\ref{hkd} that the corresponding annihilation algebra is isomorphic to one of the linearly compact Lie algebras $W_N, S_N, K_N, P_N$, where $P_N$ is the unique irreducible central extension of $H_N$. All such Lie algebras admit a $\ZZ$-grading:
\begin{itemize}
\item[---] The annihilation algebra $\W \simeq W_N$ of $\Wd$ is graded in indices $\geq -1$ by the semisimple adjoint action of the Euler vector field $\E \in \W$. 
\item[---] The inclusion $\Sd \hookrightarrow \Wd$ induces an embedding of annihilation algebras $\S \hookrightarrow \W$. Then $\S \simeq S_N$ is graded in indices $\geq -1$ by the semisimple action of the non-inner derivation $\ad \E$, where $\E \in \W$ is as above.
\item[---] The inclusion $\Kd \hookrightarrow \Wd$ from \eqref{HKembeds} induces an embedding of annihilation algebras $\K \hookrightarrow \W$. Then $\K \simeq K_N$ is graded in indices $\geq -2$ by the semisimple adjoint action of the contact Euler vector field $\E' \in \K$.
\item[---] The inclusion $\Hd \hookrightarrow \Wd$ from \eqref{HKembeds} induces a non-injective homomorphism of annihilation algebras $\P \hookrightarrow \W$, whose kernel equals the center $Z(\P)$. Then $\P \simeq P_N$ is graded in indices $\geq -2$ by the semisimple action of any lifting of the non-inner derivation $\ad \E \in \ad\W$ acting on $\P/Z(\P) =: \H \simeq H_N \subset W_N$, where $\E\in \W$ is as above.
\end{itemize}
Denote by $\L$ the annihilation algebra of our primitive Lie pseudoalgebra $\L$, and set $\L^i$ to be the graded part of degree $i$ in $\L$, with respect to the above defined grading. Then $\L = \prod_{i \geq -2} \L^i$ and
\begin{itemize}
\item $[\L^i, \L^j] \subset \L^{i+j}$;
\item $\L^0\subset \L$ is a reductive subalgebra, isomorphic to $\gl_N, \sl_N, \sp_N, \csp_{N-1} = \sp_{N-1} \oplus \kk$ when $L = \Wd, \Sd, \Hd, \Kd$ respectively;
\item each $\L^i$ is a finite-dimensional completely reducible representation of $\L^0$.
\end{itemize}
Central elements from $P_N$ all lie in the degree $-2$ part. 
We will employ the following notation:
$$\L^{> n} = \prod_{j > n}\L^j,\qquad  \L^{\geq n} = \prod_{j \geq n}\L^j,\qquad  \L^{\neq 0} = \prod_{j \neq 0} \L^j.$$
We will also denote by $\L^\perp$ the sum of all isotypical $\L^0$-components relative to nontrivial $\L^0$-actions. Then $\L^\perp\subset \L$ is a complement to the trivial isotypical $\L^0$-component, and $\L^{>0} \subset \L^\perp$. In all cases $\L$ equals its derived subalgebra, so that $\L = [\L^\perp, \L^\perp]$.
Notice that $\{\L^{\geq n}\}_{n \in \ZZ}$ is a decreasing filtration of $\L$ which coincides with $\{\L_n\}_{n \in \ZZ}$ when $\L = \Wd, \Sd, \Hd$, see \cite{BDK1, BDK3}.

In the remaining case $L = \Kd$, the two filtrations are distinct but equivalent, see \cite{BDK2}, and they induce the same linearly compact topology. More precisely,
$$\K_n \subset \K^{\geq n}, \qquad \K^{\geq n} \subset \K_{\lfloor\frac{n-1}2\rfloor},$$
so that $\K_n$ certainly contains $\K^{\geq 2n+1}$. Also notice that $\K_0 \subset \K^\perp$.

\subsection{$\Hd$, $\dd_+$ and $\spdd$}\label{shds}

Choose $r \in \bigwedge^2 \dd$, $s \in \dd$. Setting
\begin{equation}\label{freerankone}
[e*e] = (r + s\otimes 1 - 1 \otimes r) \otimes_H e
\end{equation}
defines a Lie pseudobracket on the free $H$-module $L = He$ of rank $1$ precisely when identities
\begin{align}\label{Hidentities}
[r, \Delta(s)] & = 0\\
([r_{12}, r_{13}] + r_{12}s_3) & + (\mbox{cyclic permutations}) = 0
\end{align}
hold, see \cite[Equations (4.3)-(4.4)]{BDK}.

If $r = \sum_{ij} r^{ij} \d_i \otimes \d_j$ is nondegenerate, then $(r^{ij})$ is an invertible skew-symmetric matrix, whose inverse we denote by $(\omega_{ij})$. Setting $\omega(\d_i \wedge \d_j)$ then defines a nondegenerate skew-symmetric $2$-form on $\dd$, that we may use to set $\chi = \iota_s \omega$. Then $L \simeq \Hd$ and identities \eqref{Hidentities} are equivalent to
\begin{align}
\di \omega + \chi & \wedge \omega = 0;\\
\di \chi & = 0.
\end{align}
This means that the $1$-cocycle, or {\em traceform}, $\chi$ may be used to define a $1$-dimensional $\dd$-module $\kk_\chi$, and that $\omega$ is a $2$-cocycle with values in $\kk_\chi$. In other words, we may set up an abelian extension\footnote{The Lie algebra $\dd_+$ is denoted $\dd'$ in \cite{BDK3}.}
$$0 \to \kk_\chi \to \dd_+ \stackrel{\pi}{\to} \dd \to 0,$$
whose representations prove useful in the description of $\Hd$-tensor modules. Notice that $\dd_+ \simeq \dd \oplus \kk_\chi$ as vector spaces: we shall denote the linear generator of $\kk_\chi$ by $c$ and abuse the notation by denoting with $\d$ both elements from $\dd$ and those from $\dd_+$. The Lie bracket in $\dd_+$ then satisfies
$$[\d, \d']_+ = [\d, \d'] + \omega(\d \wedge \d')c.$$
If we set $\overline{\d} = \d - \chi(\d)$, and $\d^i = \sum_{j} r^{ij}\d_j$, then
$\omega(\d^i \wedge \d_j) = \delta^{i}_{j}$ and $\d_i = \sum_j \omega_{ij}\d^j$. Also, \eqref{freerankone} rewrites as
\begin{equation}\label{freerankone2}
[e*e] = \sum_k (\overline{\d_k} \otimes \overline{\d^k}) \otimes_H e.
\end{equation}
Notice that skew-symmetry of $r$ implies $\sum_k \d_k \otimes \d^k = - \sum_k \d^k \otimes \d_k$, so that
$$\sum_k \d_k \d^k = - \sum_k \d^k \d_k = \frac{1}{2} \sum_k [\d_k, \d^k]$$
belongs to $[\dd, \dd]$. This implies that every traceform on $\dd$ vanishes on the above element.

One may use the symplectic form $\omega$ on $\dd$ to define the symplectic subalgebra $\spdd\subset \gld$. More precisely, $\phi \in \gld$ lies in $\spdd$ if and only if
$$\omega(\phi(\d) \wedge \d') + \omega(\d \wedge \phi(\d')) = 0.$$
Now define $\ad_\chi \d: \dd \to \dd$ as
$$(\ad_\chi \d)(\d') := [\d, \d'] + \chi(\d')\d.$$
We will later need the following technical fact:
\begin{lemma}\label{equivalent}
Let $\dd$ be an even dimensional Lie algebra, $\chi \in \dd^*$ a traceform on $\dd$, and assume that $2$-form $\omega \in \bigwedge^2 \dd^*$ is nondegenerate and satisfies $\di \omega + \chi \wedge \omega = 0$. 
Choose $s \in \dd$ so that $\chi = \iota_s \omega$.
If $\delta \in \dd$, then the following are equivalent:
\begin{enumerate}
\item $[s, \delta] = 0$ and $\ad_\chi \delta \in \spdd$;
\item $\iota_\delta \omega \in \dd^*$ is a traceform and $\chi(\delta) = 0$.
\end{enumerate}
\end{lemma}
\begin{proof}
We have $s = \sum_k \chi(\d_k) \d^k$ and $\chi(s) = \omega(s\wedge s) = 0$. Identity $\di \omega + \chi \wedge \omega = 0$ translates into
\begin{equation}\label{dwxw}
\omega([a, b]\wedge c) + \omega([b, c]\wedge a) + \omega([c, a]\wedge b) - \chi(a) \omega(b\wedge c) - \chi(b) \omega(c\wedge a) - \chi(c) \omega(a\wedge b) = 0, 
\end{equation}
for all choices of $a, b, c \in \dd$. Substituting $a = s$ into \eqref{dwxw} then yields
\begin{equation}\label{dwxw2}
\omega([s, b]\wedge c) +  \omega(b\wedge [s, c]) = 0, 
\end{equation}
as $\iota_s \omega = \chi$ vanishes on $[\dd, \dd]$ and
$$\chi(b) \omega(c\wedge s) + \chi(c) \omega(s\wedge b) = - \chi(b)\chi(c) + \chi(c)\chi(b) = 0.$$

(1) $\implies$ (2). If we plug $b = \delta$ into \eqref{dwxw2}, we obtain $\omega(\delta\wedge [x, c]) = 0$ for all $c \in \dd$, thus showing that the linear functional $\iota_\delta \omega \in \dd^*$ vanishes on $\im \ad s$.
Setting $a = \delta$ in \eqref{dwxw} and using $\ad_\chi \delta \in \spdd$ yields
\begin{align*}
\omega([\delta, b]\wedge c) & + \omega([b, c]\wedge \delta) + \omega([c, \delta]\wedge b)\\
 & - \chi(\delta) \omega(b\wedge c) - \chi(b) \omega(c\wedge \delta) - \chi(c) \omega(\delta\wedge b) = 0, \\
\omega([\delta, b]\wedge c) & + \omega(b\wedge [\delta, c]) + \chi(b) \omega(\delta\wedge c) + \chi(c) \omega(b\wedge \delta) = 0,
\end{align*}
which together give
\begin{equation}\label{omdbc}
\omega(\delta\wedge [b,c]) + \chi(\delta) \omega(b\wedge c) = 0.
\end{equation}
Using $b = s$ in \eqref{omdbc} then shows that $\chi(\delta)\omega(s\wedge c) = 0$ for all $c \in \dd$. As $\omega$ is nondegenerate, this forces $\chi(\delta) = 0$, whence $\omega(\delta\wedge [b,c]) = 0$ for all $b, c \in \dd$, implying $\iota_\delta \omega$ is a trace-form on $\dd$.

(2) $\implies$ (1). Substitute $b = \delta$ in \eqref{dwxw2} to obtain $\omega([x, \delta]\wedge c) + \omega(\delta\wedge [x, c]) = 0$. As $\iota_\delta \omega$ is a traceform, then $\omega([s, \delta]\wedge c) = 0$ for all $c \in \dd$. However, $\omega$ is nondegenerate, hence $[\delta, s] = 0$. Now set $a = \delta$ in \eqref{dwxw}, and use $\chi(\delta) = 0$ along with the fact that $\iota_\delta \omega$ is a traceform. This gives
\begin{equation}
\omega([\delta, b]\wedge c) + \omega([c, \delta]\wedge b) - \chi(b) \omega(c\wedge \delta) - \chi(c) \omega(\delta\wedge b) = 0. 
\end{equation}
The left hand side now rewrites as
\begin{equation}\label{dwxw3}
\omega(([\delta, b] + \chi(b))\delta\wedge c) + \omega(b\wedge ([\delta, c] + \chi(c) \delta)) = 0,
\end{equation}
thus showing that $\ad_\chi \delta \in \spdd$.
\end{proof}

We end this section by recalling notation that we will later employ. If $V$ is a representation of a Lie algebra $\g$, and $\chi \in \g^*$ is a traceform, then we will denote by $V_\chi$ the tensor product $V \otimes \kk_\chi$, where $\kk_\chi$ is the one-dimensional representation corresponding to $\chi: \g \to \kk$ viewed as a Lie algebra homomorphism. Notice that this notation is compatible with denoting by $\kk$ the trivial representation of $\g$. Also, if $\pi: \g \to \h$ is a Lie algebra homomorphism, and $\chi$ is a traceform of $\h$, we will denote by $\pi^*\chi = \chi \circ \pi$ the pullback traceform on $\g$.

Finally, if $V, W$ are Lie algebra representations of the Lie algebras $\g, \h$ respectively, then $V \boxtimes W$ is the corresponding tensor product representation of the direct sum Lie algebra $\g \oplus \h$. Notice that every irreducible $\g \oplus \h$-module is isomorphic to $V \boxtimes W$ for a suitable choice of irreducible $\g$-, resp. $\h$-, modules $V, W$.

\section{Tensor modules of primitive Lie pseudoalgebras}

Each finite irreducible representations of a primitive Lie pseudoalgebra $L$ as in Examples \ref{wd}-\ref{hkd} arises as a quotient of an opportune $L$-module from 
a special class of representations called {\em tensor modules}, that are parametrized by (finite-dimensional irreducible) Lie algebra representations of the direct sum of $\L^0$ and a Lie algebra isomorphic either to $\dd$ in types W, S, K, or $\dd_+$ in type H. this section, we recall their definition from \cite{BDK1, BDK2, BDK3}. As usual, $\dim \dd = N$ and $\{\d_1, \dots, \d_N\}$ is a basis of $\dd$.

\subsection{Tensor modules for $\Wd$}\label{wdtm}

When $L = \Wd$, the degree zero component $\L^0$ of the annihilation algebra is isomorphic to $\gld\simeq \gl_N$. Let $\{\d_i|1 \leq i \leq N\}$ be a basis of the Lie algebra $\dd$ and $R = (\Pi \boxtimes U, \rho)$ be an irreducible representation of $\dd \oplus \gld$.
Then there is a Lie pseudoalgebra action of $\Wd = H \otimes \dd$ on $H \otimes R$ defines by 
\begin{align}\label{wdtensoraction}
(1 \otimes \d_i) * (1 \otimes v) = \sum_{j = 1}^N (\d_j \otimes 1) & \otimes_H (1 \otimes \rho(e_i^j)v)\\
+\,\, (1 \otimes 1) & \otimes_H (1 \otimes \rho(\d_i + \ad \d_i)v - \d_i \otimes v),\nonumber
\end{align}
where $v \in R$ and $e_i^j \in \gld$ denotes the elementary matrix satisfying $e_i^j(\d_j) = \d_i$. The notation for this $\Wd$-module from \cite[Definition 6.2]{BDK1} is $\V(R) = \V(\Pi, U)$, but we will add a W superscript to distinguish it from tensor modules over primitive Lie pseudoalgebras of other types.

Notice that $\VW(\Pi, U)$ is irreducible unless $U$ is isomorphic to $\bigwedge^n \dd$ for some $0 \leq n <N$, in which case it has a unique nontrivial maximal $\Wd$-submodule. The corresponding quotient is thus a finite irreducible $\Wd$-module, with the single exception of the case $n=0$, when it has a trivial $\Wd$-action.

\subsection{Tensor modules for $\Sd$}\label{sdtm}

In this case, $\L^0\simeq \sld\simeq \sl_N$. Recall that $\Sd$ has a unique Lie pseudoalgebra embedding in $\Wd$, so that every $\Wd$-tensor modules as in \eqref{wdtensoraction} 
becomes an $\Sd$-module by restriction. Then a representation of $\Sd$ is a tensor module if it is the restriction of a $\Wd$-tensor module; notice that each $\Sd$-tensor module arises as such a restriction in more than one way, and that a tensor module is reducible if and only if it is the restriction of at least one reducible $\Wd$-tensor module.

If $R = (\Pi \boxtimes U, \rho)$ is an irreducible representation of $\dd \oplus \sld$, we denote by $\VS(R) = \VS(\Pi, U)$ the tensor module\footnote{This is denoted $\V_\chi(R) = \V_\chi(\Pi, U)$ in \cite[Definition 7.2]{BDK1}.} obtained by restricting the $\Wd$-tensor module $\VW(\Pi, U^0)$, where $U^0$ is the $\gld$-module obtained by extending  the $\sld$-action on $U$ so that $\id \in \gld = \sld \oplus \kk \id$ acts trivially.

Once again, the $\Sd$-tensor module $\VS(\Pi, U)$ is irreducible unless $U$ is  either trivial or isomorphic to one of the fundamental representations $\bigwedge^n \dd, 0 < n < N$. In the latter cases, there is a unique maximal submodule which yields an irreducible quotient, whereas when $U$ is the trivial $\sld$-module, the corresponding quotient has a trivial $\Sd$-action.

\subsection{Tensor modules for $\Kd$}\label{kdtm}

If $\dd_0 = \ker \theta$, then $\di \theta$ is a symplectic form on $\dd_0$, and $\L^0 \simeq \csp(\dd_0, \di \theta) = \sp(\dd_0, \di \theta) \oplus \kk c \simeq \sp_{N-1} \oplus \kk$. Here, elements $\d_1, \dots, \d_{N-1}$ form a basis of $\dd_0$ and $\d_N = s$ as in \eqref{hkbracket}. Elements $\d^k$ are defined as in Section \ref{shds}. We similarly raise indices by introducing elements $e^{ij} \in \gl \dd_0$ that satisfy $e^{ij}\d_k = \delta^j_k \d^i$. The Lie subalgebra $\sp(\dd_0, \di \theta)$ is then generated by elements
$$f^{ij} = - \frac{1}{2}(e^{ij} + e^{ji}), 1 \leq i \leq j \leq N-1.$$
Notice that differences $e^{ij} - e^{ji}$ span a (skew-symplectic) complement to $\sp(\dd_0, \di \theta)$ in $\gld_0$, so that one may consider the projection $\pi^\sp$ to the symplectic summand. Then we denote by $\ad^\sp(\d)$ the expression $\pi^\sp(\ad \d - \d_N \cdot \iota_\d \omega)$.

If $R = (\Pi \boxtimes U, \rho)$ is a finite-dimensional irreducible representation of $\dd \oplus \sp(\dd_0, \di \theta)$ then
\begin{align}
e * (1 \otimes v) = \sum_{i,j=1}^{2n} (\d_i \d_j \otimes 1) & \otimes_H( 1 \otimes \rho(f^{ij})v)\\
- \sum_{k = 1}^{N-1} (\d_k \otimes 1) & \otimes_H (1 \otimes \rho(\d^k + \ad^\sp(\d^k))v - \d^k \otimes v)\nonumber\\
+ \,\frac{1}{2}(\d_{N} \otimes 1) & \otimes_H (1 \otimes \rho(c)v)\nonumber\\
+ \,\,(1 \otimes 1) & \otimes_H (1 \otimes \rho(\d_N + \ad \d_N)v - \d_N \otimes v),\nonumber
\end{align}
where $v \in R$ and $c$ denotes the central element in $\csp(\dd_0, \di \theta) = \sp(\dd_0, \di\theta) \oplus \kk c$, defines a Lie pseudoalgebra action of $L = \Kd = He$ on $H \otimes R$, that we will denote\footnote{This is denoted $\V(R)= \V(\Pi, U)$ in \cite[Definition 5.3]{BDK2}.} by $\VK(R) = \VK(\Pi, U)$. The $\Kd$-tensor module $\VK(\Pi, U)$ is irreducible unless either the $\csp(\dd_0, \di \theta)$-action on $U$ is trivial; or $U$ is the $p$-th fundamental representation of $\sp(\dd_0, \di \theta)$ and $c$ acts via scalar multiplication by either $p$ or $N+1-p$. Whenever $\VK(\Pi, U)$ is reducible, it has a unique maximal $\Kd$-submodule yielding an irreducible $\Kd$-module, unless when $U$ is the trivial $\sp(\dd_0, \di\theta)$-module, in which case the $\Kd$-action on the quotient is trivial.

\subsection{Tensor modules for $\Hd$}\label{hdtm}

As $\chi$ is a $1$-cocycle on $\dd$, we set $\overline{\d} = \d - \chi(\d)$. Then $\d \mapsto \overline{\d}$ extends to an algebra automorphism of the universal enveloping algebra $H = \ue(\dd)$.

Here $\dd_+ = \dd + \kk c$ is the abelian extension of $\dd$ corresponding to the $2$-cocycle $\omega$ with values in $\kk_\chi$, and $\L^0 \simeq \spdd \simeq \sp_N$. For each choice of $\d \in \dd$, we denote by $\ad_\chi \d\in \gld$ the map $x \mapsto [\d, x] + \chi(x)\d$; we also set\footnote{This is denoted $\ad^\sp \d$ in \cite{BDK3}.} $\ad_\chi^\sp(\d) = \pi^\sp(\ad_\chi \d)$. If $R = (\Pi_+ \boxtimes U, \rho)$ is a finite-dimensional irreducible representation of $\dd_+ \oplus \spdd$, then
\begin{align}\label{hdtensoraction}
e * (1 \otimes v) = \sum_{i,j = 1}^{N} (\overline{\d_i \d_j} \otimes 1) & \otimes_H (1 \otimes \rho(f^{ij})v)\\
- \sum_{k = 1}^N (\overline{\d_k} \otimes 1) & \otimes_H (1 \otimes \rho(\d^k + \ad_\chi^\sp(\d^k))v - \d^k \otimes v)\nonumber\\
+ (1 \otimes 1) & \otimes_H (1 \otimes \rho(c)v) \nonumber
\end{align}
defines a Lie pseudoalgebra representation of $\Hd = He$ on $H \otimes R$, that we denote\footnote{This is denoted $\V(R) = \V(\Pi_+, U)$ in \cite[Definition 6.2]{BDK3}.} by $\VH(R) = \VH(\Pi_+, U)$. The tensor module $\VH(\Pi_+, U)$ is irreducible unless:
\begin{itemize}
\item[---]
either $U$ is trivial and $\rho(c) = 0$;
\item[---]
or $U$ is one of the fundamental representations of $\spdd$.
\end{itemize}
When $\VH(\Pi_+, U)$ is reducible, it has a unique maximal submodule if $\rho(c) = 0$, when the corresponding quotient has a nontrivial $\Hd$-action unless $U$ is the trivial $\spdd$-module. When instead $\rho(c) \neq 0$, each reducible tensor modules decomposes into the direct sum of its two irreducible maximal submodules.


\section{The current functor}\label{scurrent}

Let $H \subset H'$ be our usual Hopf algebras. One may use the inclusion homomorphism $\iota: H \to H'$ to define a scalar extension functor $H' \otimes_H$ which associates with each left $H$-module $M$ the corresponding left $H'$-module $H' \otimes_H M$. As $H'$ is free as a right $H$-module, the above functor is exact.

When $L$ is an $H$-Lie pseudoalgebra, the associated left $H'$-module $H' \otimes_H L$ may be endowed with a corresponding $H'$-Lie pseudoalgebra structure, which is uniquely determined by setting
$$[1 \otimes_H a * 1 \otimes_H b] = \sum_i (f_i \otimes g_i) \otimes_{H'} (1 \otimes _H c^i),$$
whenever $a, b \in L$ and $[a * b] = \sum_i (f_i \otimes g_i) \otimes_H c^i$. This Lie $H'$-pseudoalgebra is usually denoted by $\Cur_H^{H'} L$ and called {\em current Lie pseudoalgebra} of $L$. It follows from the structure theory developed in \cite{BDK} that $\Cur_H^{H'} L$ is simple whenever $L$ is simple, and that all finite simple Lie pseudoalgebras are obtained in this way from one of the {\em primitive} Lie pseudoalgebras listed above in Section \ref{sprimitive}.

The current functor $H' \otimes_H$ may also be applied to $L$-modules, and similarly yields Lie pseudoalgebra representations of $\Cur_H^{H'} L$, as it will be made explicit in Section \ref{scurreps} below. We will later see that irreducibility of modules behaves well with respect to the current functor.

\begin{remark}
Recall that the injection $\iota: H \to H'$ is pure, so that
$$M\ni m \mapsto 1 \otimes_H m \in \Cur_H^{H'} M$$
is always injective. In other words, every left $H$-module $M$ embeds $H$-linearly into its current module $\Cur_H^{H'}M$
\end{remark}

\subsection{Current simple Lie pseudoalgebras}

Let $L$ be a Lie pseudoalgebra over $H$, and $L' = \Cur_H^{H'} L := H' \otimes_H L$ its current pseudoalgebra over $H'$. Then $\L = X \otimes_H L$ is the annihilation algebra of $L$ and
$$\L' = X' \otimes_{H'}(H' \otimes_H L) \simeq X' \otimes_H L $$ is the annihilation algebra of $L'$; one has therefore a natural surjection
$$\iota^* \otimes_H \id_L : \L' \simeq X' \otimes_H L \to X \otimes_H L \simeq \L.$$

Recall that the Lie bracket \eqref{annihbracket} of a Lie pseudoalgebra over $H$ is obtained by rephrasing the Lie pseudobracket in terms of the right action of $H$ on its dual $X$. As the right $\dd$-action on $X \subset X'$ coincides with the natural right $\dd$-action on $X$, then the map
$$\L \simeq X \otimes_H L \to X' \otimes_H L \simeq \L'$$
is a continuous Lie algebra homomorphism and
provides a(n injective) splitting $\varsigma \otimes_H \id_L$ to $\iota^* \otimes_H \id_L$. Notice that $X' = \O\otimes X$ and we may use \prref{dactonxprime}\,{\em (ii)} in order to conclude that the Lie bracket on $\L'$ extends $\O$-linearly the Lie bracket on the subalgebra $\L \subset \L'$.

\begin{lemma}\label{trivialL}
Let $L$ be a primitive simple Lie pseudoalgebra of vector fields, and denote by
$\L^0$ denote the reductive Lie subalgebra of the annihilation algebra $\L$ consisting of degree $0$ element according to the grading recalled in Section \ref{nonontrivial}. Then the only trivial summands for the adjoint action of $\L^0$ on $\L$ are
\begin{itemize}
\item[---] $\kk \E$ when $L = \Wd$;
\item[---] the center $Z(\L)$ when $L = \Hd$;
\item[---] $\kk \E'$ when $L = \Kd$,
\end{itemize}
whereas there is no such trivial summand when $L = \Sd$. 
\end{lemma}
\begin{proof}
When $L = \Wd$, resp. $\Kd$, then the grading on $\L$ is given by the eigenspace decomposition of the inner derivation $\ad \E$, resp. $\ad \E'$. Consequently, all trivial summands for the adjoint action of $\L^0$ must lie in degree $0$. As $\L^0 \simeq \gl_N = \sl_N \oplus \kk$, resp. $\csp_{N-1} =\sp_{N+1} \oplus \kk$, the adjoint action of $\L^0$ is trivial on $Z(\L^0)$, that is on $\kk \E$, resp. $\kk \E'$.

When $L = \Hd$, \cite[Lemma 6.7 (iii)]{BDK} shows that each graded summand $\L^n \subset \L$ is an irreducible representation of $\L^0$, which is only trivial when $n = -2$. Notice that $\L^{-2}$ coincides with the center $Z(\L)$, as the $\L^0$ action on $Z(\L)$ is trivial and $\dim Z(\L) = \dim \L^{-2}$.

Absence of trivial $\L^0$-summand in the $L = \Sd$ case follows from \cite[Lemma 6.7 (ii)]{BDK}.
\end{proof}

As we have identified $\L$ with a subalgebra of $\L'$, $\L^0 \subset \L$ also arises as a subalgebra of the annihilation algebra $\L'$ of the current Lie pseudoalgebra $L' = \Cur_H^{H'} L$. In this setting
\begin{proposition}\label{onlytrivial}

The trivial isotypical component corresponding to the adjoint action of $\L^0 \subset \L \subset \L'$ on $\L'= \O \,{\widehat \otimes}\,\L$ is
\begin{itemize}
\item[---] $\O \otimes \kk\E$, where $\E \in \W$ is the Euler vector field, when $L = \Wd$; moreover $\m \otimes \kk\E \subset [\O \,{\widehat \otimes}\,\W^{>0}, \m \,{\widehat \otimes}\, \W^\perp]$.
\item[---] $(0)$ when $L = \Sd$.
\item[---] $\O \otimes \kk\E'$, where $\E' \in \K$ is the contact Euler vector field, when $L = \Kd$; moreover $\m \otimes \kk\E' \subset [\O \,{\widehat \otimes}\,\K^{>0}, \m \,{\widehat \otimes}\, \K^\perp]$.
\item[---] $\O \otimes Z(\P)$ when $L = \Hd$; moreover $\m^2 \otimes Z(\P) \subset [\m \,{\widehat \otimes}\, \P^\perp, \m \,{\widehat \otimes}\, \P^\perp]$.
\end{itemize}
\end{proposition}
\begin{proof}
The Lie bracket on $\L' = \O \,{\widehat \otimes}\, \L$ is $\O$-bilinear, so that if $X \subset \L$ is a non-trivial irreducible $\L^0$-summand, then $\O \,{\widehat \otimes}\, X$ is $X$-isotypical with respect to the action of $\L^0 \subset \L \subset \L'$. Then the description of the trivial isotypical component follow from Lemma \ref{trivialL}.


As for $\m \otimes \kk \E \subset [\O \,{\widehat \otimes}\,\L^{>0}, \m \,{\widehat \otimes}\, \L^\perp]$ when $L = \Wd$, resp. $\Kd$, it suffices to prove that $\E$, resp. $\E'$, lies in $[\L^{>0}, \L^\perp]$. This follows from the last claim in each of \cite[Lemmas 6.5, 6.7]{BDK}, as the semisimple Lie subalgebra of $\L$ denoted there by $\p$ equals its derived subalgebra.

Lastly, when $L = \Hd$, $\L = \P \simeq P_N$ is an irreducible central extension of $\P/Z(\P) = \H \simeq H_N$, so that $\P$ equals its derived subalgebra. However \cite[Lemma 6.7]{BDK} states that $\P^\perp = \P^{\geq -1}$ is a complement, as vector spaces, to the center $Z(\P)$, so that $[\P^\perp, \P^\perp]$ contains $Z(\P)$. The claim $\m^2 \otimes Z(\P) \subset [\m \,{\widehat \otimes}\, \P^\perp, \m \,{\widehat \otimes}\, \P^\perp]$ now follows by $\O$-bilinearity of the Lie bracket of $\L' \simeq \O \,{\widehat \otimes}\, \P$.
\end{proof}

A further description of $\L'$, which is valid for all current Lie pseudoalgebra and not only for simple ones, is in order. We have set up an explicit splitting $\varsigma: X \to X'$ to $\iota^*: X' \to X$. As $I := \ker \iota^*$, we obtain the direct sum decomposition $X' = \varsigma(X) \oplus I$ as right $H$-modules. Then correspondingly
$$\L' = (X+I) \otimes_H L = X \otimes_H L \oplus I \otimes_H L \simeq \L + \I,$$
where we have set $\I := I \otimes_H L$. Under the identification $\L' \simeq \O \,{\widehat \otimes}\, \L$, one has $\I = \m \,{\widehat \otimes}\, \L$ so that this reads as $\L' \simeq \kk \,{\widehat \otimes}\, \L + \m \,{\widehat \otimes}\, \L = \L + \m \,{\widehat \otimes}\, \L$.

In other words, $\L'= \L \sd \I$ is the semidirect sum of the subalgebra $\L$ with the ideal $\I$, yielding the isomorphism $\L'/\I \simeq \L$ of topological Lie algebras. By Proposition \ref{dactonxprime}\,{\em (iii)}, we argue that the above isomorphism generalizes to extended annihilation algebras $\widetilde \L' /\I \simeq \widetilde \L$.

\subsection{Current representations}\label{scurreps}

Let $L$ be a Lie pseudoalgebra over $H$ and $*: L \otimes V \to (H \otimes H) \otimes_H V$ be a pseudoaction of $L$ on the left $H$-module $V$. 
If $H \subset H'$, then we may construct both the current Lie pseudoalgebra $L' = \Cur_H^{H'} L$ and the corresponding current $H'$-module $V' = \Cur_H^{H'} V := H' \otimes_H V$.
\begin{proposition}\label{currcorrespond}
There exists a unique pseudoaction of $L'$ on $V'$ extending $H'$-bilinearly 
\begin{equation}\label{curract}
(1 \otimes_H a) * (1 \otimes v) = \sum_i (f_i \otimes g_i) \otimes_{H'} (1 \otimes_H v^i),
\end{equation}
where $a \in L, v \in V, f_i, g_i \in H$ and
\begin{equation}\label{act}
a*v = \sum_i (f_i \otimes g_i) \otimes_H v^i.
\end{equation}
Then \eqref{curract} defines a Lie pseudoalgebra representation of $L'$ on $V'$ if and only if \eqref{act} gives a Lie pseudoalgebra representation of $L$ on $V$.
\end{proposition}
\begin{proof}
The fact that if $V$ is a representation of $L$ then $V'$ is a representation of $L'$ is a trivial check. The converse depends on the fact that 
$$(H' \otimes H' \otimes H') \otimes_{H'} V' \simeq (H' \otimes H' \otimes H') \otimes_{H'} (H' \otimes_H V) \simeq (H' \otimes H' \otimes H') \otimes_{H} V,$$
and the linear isomorphism
$$(H' \otimes H' \otimes H') \otimes_{H'} V' \mapsto H' \otimes H' \otimes V'$$ 
described in Corollary \ref{pretechnical},
under such identification, restricts on $(H \otimes H \otimes H) \otimes_{H'} V$ to the analogous isomorphism
$$(H \otimes H \otimes H) \otimes_H V \to H \otimes H \otimes V$$
for $V$, so that checking that \eqref{curract} satisfies the axioms of a Lie pseudoalgebra representation --- on the set of $H'$-linear generators of the form $1 \otimes_H a, a \in L,$ and $1 \otimes_H v, v\in V$ --- is the same as checking that \eqref{act} does. However the axioms for an $H'$-Lie pseudoalgebra representations are invariant under $H'$-linear combination, and we are done.
\end{proof}

\begin{remark}\label{curroffree}
In simpler words, if $V$ is a Lie pseudoalgebra representation of $L$, the $\Cur_H^{H'}L$-module $\Cur_H^{H'}V$ is obtained via extending by $H'$-bilinearity  the pseudoaction \eqref{curract} obtained by replacing $\otimes_H$ with $\otimes_{H'}$ in each $a*v$, where $a \in L, v \in V$.
\end{remark}

If we choose, as usual, a basis $\{\d_1, \dots, \d_r, \d_{r+1}, \dots, \d_{N+r}\}$ of $\dd'$ in such a way that $\{\d_{r+1}, \dots, \d_{N+r}\}$ is a basis of $\dd \subset \dd'$, then we may use the PBW basis of $H'$ to uniquely express each element $v \in \Cur_H^{H'} V$ as
$$v = \sum_{K \in \ZZ^r\times \{0\}\subset \ZZ^{N+r}} 
\d^{(K)} \otimes_H v_K,$$
as the $\dd$-components in elements of the PBW basis may be moved on the right of $\otimes_H$ and $H'$ is a free right $H$-module generated by elements $\d^{(K)}, 
 K \in \ZZ^r\times \{0\}\subset \ZZ^{N+r}$. Notice that coefficients $v_K\in V$ depend on the choice of the basis, but their $H$-linear span only depends on $v$.

If $M \subset V'$ is an $L'$-submodule, denote now by $M_0$ the $H$-submodule of $V$ generated by all coefficients $m_K$ of elements $m \in M$.
\begin{lemma}\label{coeffM}
$M_0$ is an $L$-submodule of $V$ satisfying $H' \otimes_H (L\cdot M_0) \subset M \subset H' \otimes_H M_0$.
\end{lemma}
\begin{proof}
Every $m \in M$ is an $H'$-linear combination of its coefficients, thus $M \subset H' \otimes_H M_0$. As for the other inclusion, express $m \in M$ in the form
$$m = \sum_{L \in \ZZ^r\times {0}\subset \ZZ^{N+r}} \d^{(L)} \otimes_H m_L,$$  
where $L = (l_1, \dots, l_r, 0, \dots, 0) \in \ZZ^{N+r}$ and $m_L \in V$.
If $a \in L$ satisfies
$$a*m_{K} = \sum_{K\in 0 \times \ZZ^{N} \subset \ZZ^{N+r}} (\d^{(K)}\otimes 1) \otimes_H u_{K+L},$$
where, as before, $K = (0, \dots, 0, k_{r+1}, \dots, k_{N+r})$, we want to show that all elements of the form $1 \otimes_H u_{K+L}$ lie in $M$. This follows immediately by computing
\begin{equation*}
(1 \otimes a) * m = \sum_{K,L} (\d^{(K)} \otimes \d^{(L)}) \otimes_{H'} (1 \otimes_H u_{K+L}),
\end{equation*}
and using Lemma \ref{RS}, along with Remark \ref{RSremark}.

This shows that $1 \otimes_H (L\cdot M_0) \subset M$, so that also $H' \otimes_H (L\cdot M_0) \subset M$. However $1 \otimes_H (L \cdot M_0) \subset M$ implies $L \cdot M_0 \subset M_0$, so that $M_0 \subset V$ is an $L$-submodule.
\end{proof}

\begin{corollary}\label{currirrisirr}
Let $V$ be an irreducible $L$-module with a non-trivial $L$-action. Then $V' = \Cur_H^{H'} V$ is an irreducible $L' = \Cur_H^{H'} L$-module.
\end{corollary}
\begin{proof}
Let $(0) \neq M \subset V'$ be a submodule, $M_0\subset V$ its coefficient submodule. Then $M_0$ is a nonzero $L$-submodule of $V$, so that $M_0 = V$. Thus, by Lemma \ref{coeffM}, $H' \otimes_H (L \cdot V) \subset M \subset H' \otimes_H V$. However, $L \cdot V \subset V$ is an $L$-submodule of $V$, which cannot equal $(0)$ as the $L$-action of $V$ is non-trivial. Then $L \cdot V = V$ and $M = H' \otimes_H V = V'$.
\end{proof}

\begin{remark}
Notice that if $M_0 \subset V$ is a maximal $L$-submodule such that the action of $L$ on $V/M_0$ is non-trivial, then $H' \otimes_H M_0$ is maximal in $\Cur_H^{H'} V = H' \otimes_H V$. Indeed, $\Cur_H^{H'} (V/M_0)$ is irreducible by the above corollary; as the functor $\Cur_H^{H'}$ is exact, it commutes with taking quotient.  Then $\Cur_H^{H'} (V/M_0) \simeq V'/(H' \otimes_H M_0)$, whence maximality of $H' \otimes_H M_0$ in $V'$.
\end{remark}

\begin{corollary}\label{maxstructure}
Let $M \subsetneq \Cur_H^{H'} V = V'$ be a maximal $L'$-submodule. Then
\begin{itemize}
\item[---] either $M_0 \subset V$ is a maximal $L$-submodule and $M = H' \otimes_H M_0$;
\item[---] or $M_0 = V$ and the action of $L'$ on $V'/M$ is trivial.
\end{itemize}
\end{corollary}
\begin{proof}
If $M_0 \subset V$ is a proper non-maximal $L-$submodule, then we may locate $M_0 \subsetneq M_1 \subsetneq V$, and $M \subset H' \otimes_H M_0 \subsetneq H' \otimes_H M_1 \subsetneq V'$, thus showing $M$ is not maximal.

Viceversa, assume that $M_0 \subsetneq V$ is a maximal $L-$submodule. Then $M \subset H' \otimes_H M_0 \subsetneq V$. By maximality of $M$, we get $M = H' \otimes_H M_0$. If instead $M_0 = V$, then $H' \otimes_H (L \cdot V) \subset M \subsetneq V'$. Then $L \cdot V \subsetneq V$, so that the action of $L$ on $V/L \cdot V$ is trivial. However, $V'/M$ is a quotient of $H' \otimes_H V/L \cdot V = \Cur_H^{H'} V/L \cdot V$ showing the action of $L'$ on $V'/M$ is trivial.
\end{proof}

Our goal in this paper is figuring out to what extent the only irreducible representations of $\Cur_H^{H'} L$ are obtained by taking currents of irreducible representations of $L$, when $L$ is a simple primitive Lie pseudoalgebra.

\section{An exceptional representation of $\Cur_H^{H'} \Hdzero$}\label{Hdexceptional}

Let $L' = \Cur_H^{H'}L$ where $L = \Hd = He$ satisfies
\begin{equation}
[e * e] = \sum_k (\overline{\d_k} \otimes \overline{\d^k}) \otimes_H e.
\end{equation}
Here $\overline{\d} = \d - \chi(\d)$, and we will be using the notation introduced in Sections \ref{shds} and \ref{hdtm}. Notice that $L' = H' \otimes_H He$ may (and will) be identified with $H'e$.

Now assume that $R = (\Pi_+ \boxtimes U, \rho)$ is a finite-dimensional irreducible $\dd_+ \oplus \spdd$-module. We aim to find conditions characterizing the values of $t \in \dd' \setminus \dd$ making the following modification of the tensor module pseudoaction given in \eqref{hdtensoraction}:
\begin{align}\label{modifiedaction}
e *_t (1 \otimes v) & = \sum_{ij} (\overline{\d_i \d_j} \otimes 1) \otimes_{H'} ((1 \otimes \rho(f^{ij})v)\\
& - \sum_k (\overline{\d_k} \otimes 1) \otimes_{H'} (1 \otimes \rho(\d^k + \ad^\sp \d^k)v - \d^k \otimes v)\nonumber\\
& + (1 \otimes 1) \otimes_{H'} (1 \otimes \rho(c)v),\nonumber\\
& + (t \otimes 1) \otimes_{H'} (1 \otimes v),\nonumber
\end{align}
into a Lie pseudoalgebra representation of $L' = \Cur_H^{H'} L$. Clearly, removing the last summand in the right-hand side yields the current $L'$-module $\Cur_H^{H'} \VH(R)$.

In analogy with Section \ref{hdtm}, use the definition $(\ad_\chi x)(\d):= [x, \d] + \chi(\d)x$ to extend $\ad_\chi$ to a map $\dd' \to \Hom(\dd, \dd')$. Recall that $s = \sum_k \chi(\d_k)\d^k$.
\begin{proposition}\label{twisted}
Let $t$ be an element of $\dd' \setminus \dd$. Expression \eqref{modifiedaction} defines a Lie pseudoalgebra representation of $L'$ precisely when $[t, s] = 0$ and $\ad_\chi t$ preserves $\dd$ and lies in $\spdd$.
\end{proposition}
\begin{proof}
Explicitly compute
$$e *_t (e *_t (1 \otimes v)) - (\sigma \otimes \id_H) \otimes_H e *_t (e *_t (1 \otimes v)) - [e * e] *_t (1 \otimes v).$$
In order for $*_t$ to be a Lie pseudoalgebra action of $\Hd$, this must vanish. All terms indeed cancel out, with the exception of
\begin{equation}\label{chineqzero}
\sum_k \left(\overline{\d_k} \otimes ([t, \d^k] + \chi(\d^k)t) + ([t, \d_k] + \chi(\d_k)t) \otimes \overline{\d^k}\right) = 0.
\end{equation}
We may rewrite \eqref{chineqzero} as
\begin{align}
0 = & \sum_k \left(\d_k \otimes ([t, \d^k] + \chi(\d^k)t) + ([t, \d_k] + \chi(\d_k)t) \otimes \d^k\right)\label{chineqzero2a}\\
& - 1 \otimes \sum_k \chi(\d_k)\cdot([t, \d^k] + \chi(\d^k)t)\label{chineqzero2b}
\\
 - & \sum_k \chi(\d^k)\cdot([t, \d_k] + \chi(\d_k)t) \otimes 1.
 \nonumber
\end{align}
The three summands lie in $\dd' \otimes \dd', \kk \otimes \dd', \dd'\otimes \kk$ respectively, so they need to vanish separately.
The summand in \eqref{chineqzero2a} may be rewritten as
\begin{align}
& 0 = \sum_k \left(\d_k \otimes (\ad_\chi t)(\d^k) + (\ad_\chi t)(\d_k)\otimes \d^k\right)\label{chineqzero3a},
\end{align}
showing $\ad_\chi t$ preserves $\dd$ and lies in $\spdd$. 
The summand in \eqref{chineqzero2b} only vanishes if
\begin{equation}\label{chineqzero3b}
[t, \sum_k \chi(\d_k)\d^k] + \sum_k \chi(\d_k)\chi(\d^k)t = 0.
\end{equation}
However the second term must vanish as $\sum_k \d_k \otimes \d^k = - \sum_k \d^k \otimes \d_k$ so that \eqref{chineqzero2b} translates into $[t, s] = 0$. The third summand is now dealt with similarly and cancels out.
\end{proof}

\begin{remark}
When $\chi = 0$, the above conditions translate into the fact that $s$ normalizes $\dd$ and $\ad s \in \spdd$.
\end{remark}

\begin{example}
Let $\dd \subset \dd'$ be abelian Lie algebra, $\chi = 0$. Then $[t, s] = 0$ and $\ad_\chi t = 0 \in \spdd$ hold for every $t \in \dd'$. In particular \eqref{modifiedaction} defines a Lie pseudoalgebra representation for every $t \in \dd' \setminus \dd$.
\end{example}

\begin{example}
Let $h, e, f$ be the standard generators of $\dd' = \sl_2$, and choose $\dd = \langle h, e\rangle$. Then \cite[Example 8.14]{BDK} shows that the $H$-linear span of
$$(2-h) \otimes e + e \otimes h \in \Wd$$
is a subalgebra isomorphic to $\Hd$, where $\omega(e \wedge h) = 1$ and $\chi = \iota_{2e} \omega$.

The only elements in $\dd'$ commuting with $s = 2e$ are multiples of $e$, so there is no $t \in \dd' \setminus \dd$ satisfying $[t, s] = 0$. Consequently, \eqref{modifiedaction} does not define a Lie pseudoalgebra representation for any choice of $t \in \dd' \setminus \dd$, and all finite irreducible $\Cur_H^{H'} \Hd$-modules are, in this case, current representations. 
\end{example}

Whenever $t\in \dd' \setminus \dd$ satisfies the conditions of Proposition \ref{twisted}, we shall denote by $\VHtR = \VHtPU$ the corresponding Lie pseudoalgebra representation of $\Cur_H^{H'} \Hd$.

\begin{theorem}\label{nontrivialiso}
Let $t, t' \in \dd'$ be elements satisfying the conditions of Proposition \ref{twisted} and $\Pi_+, U$ be representations of the Lie algebras $\dd_+, \spdd$ respectively. Also denote by $\pi: \dd_+ \to \dd$ the canonical projection. Then $\VHtprime(\Pi_+, U)$ and $\VHt(({\Pi_+})_{\pi^*(\iota_{t'-t}\omega)}, U)$ are isomorphic $\Cur_H^{H'}\Hd$-modules.
\end{theorem}
\begin{proof}
Set $t' = t + \delta$. As $[s, t] = [s, t'] = 0$, then $[s, \delta] = 0$. Also, $\ad_\chi t'$ and $\ad_\chi t$ both lie in $\spdd$, hence $\ad_\chi \delta \in \spdd$ by linearity of $\ad_\chi$. Thus $s$ satisfies conditions (1) of Lemma \ref{equivalent}, which forces $\iota_\delta \omega$ to be a traceform of $\dd$ and $\chi(\delta) = 0$.
As
$$\delta = \sum_k \omega(\delta\wedge \d_k) \d^k = - \sum_k \omega(\delta\wedge \d^k) \d_k,$$
then
\begin{align}
(\delta & \otimes 1) \otimes_{H'} (1 \otimes v) = - \sum_{k} (\d_k \otimes 1) \otimes_{H'} ((\iota_\delta \omega)(\d^k) \otimes v)\\
& = - \sum_k ((\overline{\d_k} + \chi(\d_k)) \otimes 1) \otimes_{H'} ((\iota_\delta \omega)(\d^k) \otimes v)\nonumber\\
& = - \sum_k (\overline{\d_k} \otimes 1) \otimes_{H'} ((\iota_\delta \omega)(\d^k) \otimes v) - (1 \otimes 1) \otimes_{H'} (((\iota_\delta \omega)(\sum_k \chi(\d_k) \d^k)) \otimes v)\nonumber\\
& = - \sum_k (\overline{\d_k} \otimes 1) \otimes_{H'} ((\iota_\delta \omega)(\d^k) \otimes v),\nonumber
\end{align}
as
$$(\iota_\delta \omega)(\sum_k \chi(\d_k) \d^k) = \omega(\delta\wedge s) = - \chi(\delta) = 0.$$
We may now rewrite the action of $\Cur_H^{H'}\Hd$ on $\VHtprime(\Pi_+, U)$ as follows:
\begin{align}
e & *_{t'} (1 \otimes v) = e *_t (1 \otimes v) + (\delta \otimes 1) \otimes_{H'} (1 \otimes v)\\
& = e *_t - \sum_k (\overline{\d_k} \otimes 1) \otimes_{H'} ((\iota_\delta \omega)(\d^k) \otimes v)\nonumber\\
& =  \sum_{ij} (\overline{\d_i \d_j} \otimes 1) \otimes_{H'} ((1 \otimes f^{ij}.v)\nonumber\\
& - \sum_k (\overline{\d_k} \otimes 1) \otimes_{H'} (1 \otimes ((\d^k.v + (\iota_\delta \omega)(\d^k)\,\, v) + (\ad^\sp \d^k).v)) - \d^k \otimes v)\nonumber\\
& + (1 \otimes 1) \otimes_{H'} (1 \otimes c.v),\nonumber\\
& + (t \otimes 1) \otimes_{H'} (1 \otimes v).\nonumber
\end{align}
As $\iota_\delta \omega$ is a traceform on $\dd$, and the action of $c \in \dd_+$ is left unchanged, the contribution $(\iota_\delta \omega)(\d^k)$ may be absorbed in the $\dd_+$ representation, thus yielding the pseudoaction of $\Cur_H^{H'}\Hd$ on $\VHt(({\Pi_+})_{\pi^*(\iota_{\delta}\omega)}, U)$.
\end{proof}

\section{Irreducible representations of current Lie pseudoalgebras not of type H}\label{notH}

The usual strategy \cite{BDK1, BDK2, BDK3} towards describing irreducible representations of a simple Lie pseudoalgebra $L$ is by locating {\em singular vectors} for the action of the corresponding annihilation algebra, i.e., vectors that are killed by the action of $\L^{>0}$. In this paper, our point of view is that, as far as reasonable, the action of $\L'$ is recovered from that of $\L \subset \L'$, and singular vectors for the two actions coincide: in order for this philosophy to hold, one must introduce a few tweaks for current Lie pseudoalgebras of type H, so that we will treat other cases first. Throughout this section, $L$ will be a primitive Lie pseudoalgebra from Examples \ref{lie}-\ref{hkd}, and $L' = \Cur_H^{H'}$ its current pseudoalgebra.

We choose $S \subset L$, $\ell \in \NN,$ as in Section \ref{sprimitive} and endow the annihilation algebra $\L$ with the filtration $\L_n = \fil_{n+\ell-1}X \otimes_H S$. As $S \subset L \simeq 1 \otimes_{H} L \subset H' \otimes_H L = \Cur_H^{H'}L = L'$ still generates $L'$ as an $H'$-module, we may choose it in order build up a similar filtration for its annihilation algebra $\L' = X' \otimes_{H'} L'$, that we will freely identify with $X' \otimes_H L$.

Notice that   the value of $\ell\leq 2$ is the same for $L$ and $L'$ and that $\I = I \otimes_{H} L = I \otimes_{H} S \subset \L'_{1-\ell}\subset \L'_{-1}$, as $I \subset \fil_0 X'$. We will employ the notation $\I^i = I^i \otimes_{H} L = I^i \otimes_{H} S \subset \L', i \geq 0$, where $I^i$ is the $i$-fold power of the ideal $I \subset X'$, so that $\I^0 = \L', \I^1 = \I$.

The computation of the normalizer of $\L_n$ in $\widetilde \L$ is accomplished in \cite{BDK1, BDK2, BDK3}. Our present goal is to find a large subalgebra $\N'\subset \widetilde \L'$ normalizing $\L'_n, n \geq 0$. We already know that $[\L'_0, \L'_n] \subset \L'_n$.

\subsection{$\I$ normalizes $\L'_n$}

\begin{proposition}
$\I^i$ is an ideal of $\L'$. Also $[\I^i, \I^j] \subset \I^{i+j}$ and $\I^i \subset \L'_{i-s}$.
\end{proposition}
\begin{proof}
Structure constants in Lie pseudobracket of $L$ only involve elements in $H$. As $(I).\dd \subset I$, then also $(I^i).H \subset I^i$, thus proving the first two claim. The last claim is clear, as $I^i \subset \fil_{i-1} X'$.
\end{proof}

\begin{proposition}
$\I = \I^1$ normalizes $\L'_n$ for each $n$.
\end{proposition}
\begin{proof}
By Proposition \ref{lowdegree}, $S$ has been chosen in such a way that structure constants are contained in $(\dd + \kk) \otimes (\dd + \kk)$. Notice that $(I).\dd \subset I$, and $(\fil_k X').\dd \subset (\fil_k X').\dd' \subset \fil_{k-1} X$.
Thus
$$[\I, \L'_n] = [I \otimes_H S, \fil_{n+\ell-1}X' \otimes_H S] \subset  (I \cdot \fil_{n+\ell-2}X') \otimes_H S.$$
However, this lies in $\fil_{n+\ell-1}X' \otimes_H S = \L'_n$ as $I \subset \fil_0 X'$, see Remark \ref{donotadd}.
\end{proof}

\subsection{Normalizing elements not contained in $\L'$ when $L \neq \Hd$}

We will henceforth assume that $L$ is not isomorphic to $\Hd$.

We have seen above that we have a surjective homomorphisms algebras $\iota^*\otimes_H \id_L: \L' \to \L$ along with a splitting $\varsigma\otimes_H \id_L: \L \to \L'$. Moreover, $\L'$ may be identified with $\O \otimes \L$, where $\O = \kk[[x^1, \dots, x^r]]$, and the Lie bracket of $\L'$, under these identification, extends $\O$-bilinearly that of $\L$.

As $L$ is a primitive Lie $H$-pseudoalgebra, which is not isomorphic to $\Hd$, then $\L$ is simple, so that 
$$\Der (\O {\widehat \otimes} \L) = (\Der \O \otimes \id_\L) \sd (\O\, {\widehat\otimes} \Der \L),$$ see for instance \cite[Proposition 6.12(ii)]{BDK}.
Here $\L$ is the annihilation algebra of one of the primitive Lie pseudoalgebras, so $\Der \L$ equals
\begin{itemize}
\item[---] $\kk \ad \E + \ad \L$ when $L = \Sd
$.
\item[---] $\ad \L\simeq \L$ in all other cases.
\end{itemize}
Notice that $\L$ is isomorphic to $\g$, when $L$ is a simple finite-dimensional Lie algebra over $\g$, whereas $\L\simeq W_N$, (resp. $K_N$) when $L = \Wd$ (resp. $\Kd$).
Furthermore, the non-inner derivation $\E$ stabilizes $\ad \L$ and provides an explicit splitting of $\Der \L/\ad \L \simeq \kk \E$. Set $D$ to equal $\E$ when $L = \Sd$ and $0$ when $L = \Wd$ or $\Kd$. Then we may summarize the above information in the following
\begin{proposition}
Let $L\neq \Hd$ be a primitive simple Lie pseudoalgebra, $L' = \Cur_H^{H'} L$, and denote by $\L, \L'$ the corresponding annihilation Lie algebras. Then
\begin{equation}\label{sdsd}
\Der \L' = ((\Der \O \otimes \id_\L) + (\O {\otimes}  D)) \sd (\O \,{\widehat \otimes}\, \ad \L).
\end{equation}
\end{proposition}

\vspace{.5cm}

Let $\d \in \dd \subset \dd' \subset \dd \sd \L' = \widetilde \L'$. Then the adjoint action of $\d$ on $\L'$ is a derivation of $\L'$, and we may consider its projection $\pi_1(\ad \d) \in (\Der \O \otimes \id_\L) + \O \otimes  D$ on the first summand of the above decomposition. As \eqref{sdsd} is a semidirect sum, then $\pi_1$ is a Lie algebra homomorphism, whence $\dd \ni \d \mapsto \pi_1(\ad \d)$ is as well.
When $\d \in \dd$, denote now by $\widehat{\d}$ the element $\d - [\pi_2(\ad \d)] \in \widetilde{\L'}$, where we denote by $[\pi_2(\ad \d)]$ the unique element in $\L'$ inducing 
the inner derivation $\pi_2(\ad \d) \in \O \,{\widehat \otimes} \ad \L = \ad \L'$. The action of $\widehat{\d}$ on $\L'$ coincides with $\pi_1(\ad \d)$, so that
$$\eta: \dd \ni \d \mapsto \widehat{\d} \in (\Der \O \otimes \id_\L) + \O \otimes  D$$
is a Lie algebra homomorphism. We denote by $\widehat\dd$ the Lie subalgebra consisting of all $\widehat{\d}, \d \in \dd$.

\begin{remark}
Of course, when $L \simeq \g$ is a simple Lie algebra over $\kk$, then $\dd = (0)$, and $\eta$ is the trivial map.
\end{remark}

\begin{remark}
The special case $H = H'$, $L = L'$ has already been treated in \cite{BDK1, BDK2, BDK3}. Here $\O = \kk$, so that $\Der \O = 0$ and elements from $\widehat \dd$ act on $\L$ as multiples of $D$.
\begin{itemize}
\item[---] $\widehat \dd$ coincides with $\widetilde \dd\subset \widetilde{\L}$ from \cite{BDK1, BDK2} when $L$ equals either $\Wd$ or $\Kd$. As $D = 0$, the subalgebra $\widehat \dd\subset \widetilde \L$ centralizes $\L$.
\item[---] $\widehat \dd$ is the same as $\widehat \dd \subset \widetilde{\L}$ in \cite{BDK1} when $L = \Sd$.
\end{itemize}
\end{remark}


As $\I \subset \L'$ is an ideal, and the left action of $\dd$ on $\L'$ stabilizes $\I$, we conclude that all elements $\widehat \d, \d \in \dd,$ stabilize $\I$. Then

\begin{proposition}
Denote by $(\Der \O)_0$ the family of all derivations of $\O$ mapping $\O$ to its maximal ideal $\m$. Then the adjoint action of $\widehat \d\in \widetilde \L$ on $\L$ lies in $(\Der \O)_0 \otimes \id_\L + \O \otimes  D$ for all $\d \in \dd$.
\end{proposition}
\begin{proof}
The derivation $D$ stabilizes $\I = \m \,{\widehat \otimes}\, \L$. Furthermore, elements in $\Der \O \otimes \id_\L$ stabilize $\I$ iff the corresponding derivation stabilizes $I$.
\end{proof}

\begin{corollary}\label{centralizelzero}
Elements in $\widehat \dd$ normalize $\L'_n$ and centralize $\L^0$.
\end{corollary}
\begin{proof}
Identify $\L'$ with $X' \otimes_H S \simeq (\O \,{\widehat \otimes}\, X) \otimes_H S\simeq \O \,{\widehat \otimes}\, \L$. Under this identification $\L'_n = \fil_{n+\ell-1} X' \otimes_H S$ corresponds to $$\sum_{i+j = n} \m^i \,{\widehat \otimes}\,  \L_j.$$
Now, $D$ stabilizes each $\L_j$, so that $\O \otimes D$ stabilizes all of the above summands. Elements from $(\Der \O)_0$ stabilize $\m^i$ and do nothing on the tensor factor lying in $\L$, so that they also stabilize all of the above summands. As the adjoint action of $\widehat \dd$ on $\L'$ lies in $(\Der \O)_0 \otimes \id_\L + \O \otimes D$, and $D$ centralizes $\L^0$, the claim is proved.
\end{proof}

\begin{theorem}
For each $n \geq 0$, the normalizer of $\L'_n$ in $\widetilde \L'$ contains $\L^{\geq 0} + \I + \widehat \dd$.
\end{theorem}
\begin{proof}
We have proved above that the normalizer contains $\L'_0 + \I + \widehat \dd$. When $L = \Wd, \Sd$, then $\L'_0 = \sum_j \m^j \,{\widehat \otimes}\, \L_{-j}$ so that this equals $\L_0 + \I + \widehat \dd$, and $\L_0 = \L^{\geq 0}$.

When $L = \Kd$, recall that the filtrations $\{\L_n\}$ and $\{\L^{\geq n}\}$ do not coincide. However, the normalizer of $\L'_n$ also contains $1 \otimes \E'$. Now use $\L^{\geq 0} = \L_0 + \E'$ to obtain $1 \otimes \E' + \L'_0 + \I + \widehat \dd = \L^{\geq 0} + \I + \widehat \dd$.
\end{proof}

\subsection{Finite-dimensional irreducible representations of $(\L^{\geq 0} + \I + \widehat \dd)/\L'_n$}

As $\N' = \L^{\geq 0} + \I + \widehat \dd$ normalizes $\L'_n$ in $\widetilde{\L'}$, if $V'$ is an $\widetilde{\L'}$-module and $R \subset V'$ is a finite-dimensional subspace killed by $\L'_n$, then the action of $\N'$ must stabilize $R$. However, if the action of $\N'$ on $R$ is irreducible, a large part of $\N'$ will have to act trivially on $R$.

\begin{proposition}
When $n\geq 0$, the subalgebra $(\L^{>0} + \I + \L'_n)/\L'_n$ is a solvable ideal of $\N'/\L'_n$.
\end{proposition}
\begin{proof}
First of all, $$\L'_n \subset \L'_0 \subset \L_0 + \I \subset \L^{\geq 0} + \I \subset \N',$$
so that $\L'_n$ is an ideal of $\N'$. We already know that $\I$ is an ideal of the whole $\widetilde \L'$, so $\I + \L_n$ is an ideal of $\N'$ as well. Furthermore,
$[\I^j, \I^k] \subset \I^{j+k}$ forces the lower central sequence of $\I$ to lie in $\I^{n+1} \subset \L'_{n+2-\ell} \subset \L'_n$ after $n+1$ steps, so that $(\I + \L'_n)/\L'_n$ is a nilpotent, hence solvable, ideal of $\N'/\L'_n$. We are thus left with showing that $\L^{>0}$ projects to a solvable ideal of $\N'/(\I + \L'_n)$.

Now, $\L^{>0}$ is a subalgebra of $\L \subset \L'$ and the adjoint action of $\L^{\geq 0} + \I + \L'_n$ maps it to $\L^{>0} + \I + \L'_n$, whence $\L^{>0}$ projects to an ideal of $\N'/(\I + \L'_n)$. Again, $[\L^{\geq j}, \L^{\geq k}] \subset \L^{\geq j+k}$ and $\L^{\geq i} \subset \L_{\lfloor (i-1)/2\rfloor}$, forces the lower central sequence of $\L^{>0}$ to lie inside $\L_n \subset \L'_n \subset \I + \L'_n$ within $2n+1$ steps. Thus $\L'_n$ projects to a solvable ideal of $\N'/(\I + \L'_n)$, proving the claim. 
\end{proof}

\begin{proposition}\label{lposplusi}
Let $R$ be a finite-dimensional irreducible $\N'$-module, with a trivial action of $\L'_n$, where $n \geq 0$. Then $\L^{>0} + \I$ acts trivially on $R$.
\end{proposition}
\begin{proof}
We already know that $\L^{>0} + \I$ projects to a solvable ideal of the finite-dimensional Lie algebra $\N'/\L_n$. Then we may use Proposition \ref{onlytrivial} and the argument in \cite[Lemma 3.4]{BDK1} applied to both the semisimple part and the center of the reductive Lie algebra $\L^0 \subset \N$ in order to show that
$(\L^{>0} + \I) \cap \O \,{\widehat \otimes}\, \L^\perp$ acts trivially on $R$.

However, both $\L^{>0}$ and $\m \,{\widehat \otimes}\, \L^\perp$ lie inside $\O \,{\widehat \otimes}\, \L^\perp$, so they need to act trivially on $R$. Recall now that $\I = \m \,{\widehat \otimes}\, \L^\perp + \m \otimes D$. We may then use Proposition \ref{onlytrivial} to argue that elements from $\m \otimes  D \subset [\O \,{\widehat \otimes}\, \L^{>0}, \m \,{\widehat \otimes}\, \L^\perp]$ arise as Lie brackets of elements acting trivially on $R$. We thus conclude that both $\L^{>0}$ and $\I$ act trivially on $R$.
\end{proof}

\subsection{Irreducible representations and tensor modules}

As usual, $L$ is a primitive simple Lie $H$-pseudoalgebra not isomorphic to $\Hd$ and $L' = \Cur_H^{H'} L$ is the corresponding current simple Lie pseudoalgebra. 

Let $V'$ be a finite irreducible
representation of $L$, which has a nontrivial pseudoaction  by definition. Then $\ker_n V' := \{v \in V'\,|\, \L'_n v = 0\}$ is a finite dimensional vector space for all choices of $n$,
and is nonzero for sufficiently large values of $n$, as $V'$ is a discrete continuous representation of $\L'$. As $\N'$ normalizes $\L'_n$, its action on $V'$ preserves $\ker_n V'$. Pick a minimal nonzero, hence irreducibile, $\N'$-submodule $R \subset V'$. Then the action of $\N' = \L^{\geq 0} + \I + \widehat \dd$ factors through the quotient $\N'/(\L^{>0} + \I) \simeq (\widehat \dd + \L^{\geq 0})/\L^{>0}$, which is isomorphic to the direct sum Lie algebra $\dd \oplus \L^0$ by Corollary \ref{centralizelzero}. In particular, $\L'_1 \subset \L^{>0} + \I$ acts trivially on $R$, showing $R \subset \ker_1 V'$.

\begin{theorem}\label{currenttensor}
Let $V'$ be a finite irreducible representation of $L' = \Cur_H^{H'} L$, where $L\neq \Hd$ is a primitive finite simple Lie pseudoalgebra. Then there exists an irreducible finite-dimensional $\widehat \dd \oplus \L^0$-module $R$ such that $V'$ is a quotient of the current tensor module $\Cur_H^{H'} \V(R)$ by a maximal submodule, where $\V(R)$ is a tensor module for $L$ as in Sections \ref{wdtm}-\ref{kdtm} or a finite-dimensional representation of the Lie algebra $L$ when $\dd = (0)$.
\end{theorem}
\begin{proof}
We only deal with the $\dd \neq (0)$ case, as the Lie algebra case is completely analogous. Recall that $\ker_1 V'$ is finite-dimensional. If $0 \neq R \subset \ker_1 V'$ is a minimal, hence irreducible, $\N'$-submodule, then the induced module $\Ind_{\N'}^{\widetilde \L} R$ maps surjectively on $V'$ by irreducibility.
As $\widetilde \L' = \dd + \N'$, then $\Ind_{\N'}^{\widetilde \L} R \simeq \ue(\d') \otimes R$, and it is thus a finite left $H'$-module by multiplication on the first tensor factor. 

We know by Proposition \ref{lposplusi} that $\L^{>0} + \I$ acts trivially on $R$, so we may use Theorem \ref{repdescription} to write down the Lie pseudoalgebra action of $L'$ on $H' \otimes R$. This yields
\begin{align}
(1 \otimes_H s) * (1 \otimes u) & = \sum_{K\in \NN^{N+r}} (S(\d^{(K)}) \otimes 1) \otimes_{H'} (1 \otimes (x_K \otimes_{H'} s)\cdot u)\label{curraction}\\
& = \sum_{K \in 0 \times \NN^{N}} (S(\d^{(K)}) \otimes 1) \otimes_{H'} (1 \otimes (x_K \otimes_{H'} s)\cdot u),\label{curraction2}
\end{align}
where $s \in S$, $u \in R$ and $S \subset L$ is as in Section \ref{sprimitive}.
As the only $S(\d^{(K)})$ involved in \eqref{curraction2} lie in $\dd$, and the corresponding Fourier coefficients $x_K \otimes_{H'} s$ all lie in $\L \subset \L'$, we may use Proposition \ref{currcorrespond} to conclude that the pseudoaction
\begin{equation*}
s * (1 \otimes u) = \sum_{K \in \NN^{N}} (S(\d^{(K)}) \otimes 1) \otimes_{H} (1 \otimes (x_K \otimes_{H'} s)\cdot u)
\end{equation*}
defines a Lie pseudoalgebra representation of $L$ on $H\otimes R$, which coincides with the corresponding $L$-tensor module $\V(R)$. Then the $L'$-representation defined by \eqref{curraction} coincides with $\Cur_H^{H'} \V(R)$.
\end{proof}

\begin{corollary}\label{irrcurriscurrirr}
Every finite irreducible representation of $L' = \Cur_H^{H'} L$, where $L\neq \Hd$ is a primitive Lie pseudoalgebra, is of the form $\Cur_H^{H'} V$, for some finite irreducible $L$-module $V$.
\end{corollary}
\begin{proof}
By Theorem \ref{currenttensor}, every finite irreducible $L'$-module $V'$ arises as a quotient of $\Cur_H^{H'} \V(R)$ for an opportune choice of a tensor module $\V(R)$ for the primitive Lie pseudoalgebra $L$. If $\V(R)$ is irreducible, then $\Cur_H^{H'}\V(R)$ is irreducible as well by Corollary \ref{currirrisirr}. Thus $\V(R) \simeq \Cur_H^{H'}\V(R)$ and we are done.

If $\V(R)$ is not irreducible, then we may find a maximal $L'$-submodule $M \subset \Cur_H^{H'} \V(R)$ such that $V' \simeq \V(R)/M$. As the $L'$-action on $V'$ is nontrivial, we may apply Corollary \ref{maxstructure} and conclude that $M = \Cur_H^{H'} M_0$, where $M_0\subset \V(R)$ is a maximal $L$-submodule. Then $V' = (\Cur_H^{H'} \V(R)) / (\Cur_H^{H'} M_0) \simeq \Cur_H^{H'} (\V(R)/M_0)$ and $V:= \V(R)/M_0$ is an irreducible $L$-module. Finiteness, i.e, Noetherianity, of $V$ follows easily from finiteness of $V'$ and exactness of the current functor.
\end{proof}

\section{Irreducible representations of current Lie pseudoalgebras of type H}\label{shtype}

In this section, we review the strategy of Section \ref{notH} in the case of the simple Lie pseudoalgebra $\Cur_H^{H'} \Hd$, highlighting the relevant differences, which depend on the presence of nontrivial central elements in the corresponding annihilation algebras.
Once again, we employ the notation introduced in Sections \ref{shds} and \ref{hdtm}, so that $L = \Hd = He$ satisfies
$$[e * e] = \sum_k (\overline{\d_k} \otimes \overline{\d^k}) \otimes_H e.$$

Its annihilation algebra $\P \simeq P_N$ is an irreducible central extension of the simple linearly compact Lie algebra $H_N$ from the Cartan list. More precisely, $Z(\P)$ is one-dimensional, spanned by $e^{-\chi} \otimes_H e$, and $\H:=\P/Z(\P)$.

\subsection{Derivations of $\O \,{\widehat \otimes}\, P_N$}

It is well known that $\Der \H = \kk \ad \E + \ad \H$, where $\ad\E$ is the semisimple derivation of $H_N$ inducing the standard grading. As $\P\simeq P_N$ is itself a graded Lie algebra, and its grading is compatible with that of its quotient $H_N$, we will also denote by $\ad \E$ the corresponding derivation of $P_N$. Then
\begin{proposition}\label{derp}
One has $\Der \P = \kk \ad \E \sd \ad \P$. Furthermore,
$$\Der(\O \,{\widehat \otimes} \P) = (\Der \O \otimes \id_\P) \sd (\O \,{\widehat \otimes}\, \Der \P) = (\Der \O \otimes \id_\P \sd \O \otimes \ad \E) \sd \O \,{\widehat \otimes}\, \ad \P.$$
\end{proposition}

\begin{lemma}\label{dercenter}
Let $\g$ be a Lie algebra, $d\in \Der \g$ a derivation whose image is contained in the center $Z(\g)$. If $\g = [\g, \g]$, then $d = 0$.
\end{lemma}
\begin{proof}
One has $d[x, y] = [d(x), y] + [x, d(y)]$. As $\im d \subset Z(\g)$, then $d[x, y] = 0$ for every $x, y \in \g$. 
\end{proof}

\begin{proof}[Proof of Proposition \ref{derp}]
As $\H \simeq H_N$ is a linearly compact simple Lie algebra, we already know from \cite[Proposition 6.12 (ii)]{BDK} that $\Der(\O \,{\widehat \otimes} \H) = \Der \O \otimes \id_\H + \O \,{\widehat \otimes}\, \Der \H$.

Let $\delta
 \in \Der(\O \,{\widehat \otimes} \P)$. Every derivation of a Lie algebra stabilizes its center, so $\delta$ induces a derivation $\widetilde \delta$ of $(\O \,{\widehat \otimes} \P)/Z(\O \,{\widehat \otimes} \P) \simeq \O \,{\widehat \otimes} \H$. Then there exist $d\in \Der \O, \overline{w} \in \H, c \in \kk,$ such that 
$${\widetilde \delta}(\phi \otimes \overline{x}) = d(\phi) \otimes x + \phi \otimes ([\overline{w}, \overline{x}] + c\ad \E(\overline{x})),$$

Choose a lifting $w\in \P$ of $\overline{w}\in \H$. Then
$$\delta(\phi \otimes x) \equiv d(\phi) \otimes x + \phi \otimes ([w, x] + c \ad \E (x)) \mod Z(\O\,{\widehat \otimes} \P),$$
so that $\delta - (d \otimes 1 + \phi \otimes (\ad w + c \ad \E))$ is a derivation of $\O \,{\widehat \otimes} \P$ mapping everything to the center. However, $\P$ equals its derived Lie algebra, so the same holds of $\O \,{\widehat \otimes} \P$ and we may use Lemma \ref{dercenter} to conclude. The case $r=0, \O = \kk$ takes care of the first claim.
\end{proof}

\subsection{The normalizer of $\P'_n, n>0,$ in $\widetilde{\P'}$}

Let $L = \Hd$, $L' = \Cur_H^{H'} L$ and denote by $\P, \P'$ the corresponding annihilation algebras. We have seen that $\iota^*:X' \to X$ induces a projection $\P' \to \P$ which admits a splitting; we thus obtain a subalgebra $\P \subset \P'$ and the Lie bracket on $\P' \simeq \O \,{\widehat \otimes}\, \P$, where $\O = (H'H_+)^\perp \subset X'$, extends $\O$-bilinearly that of $\P$. If $I = (H'H)^\perp$, then $X' = X + I$ and $\P' = \P \sd \I$, where $\I = I \otimes_H L$. Notice that $\I$ corresponds to $\m \,{\widehat \otimes} \P$ under the identification $\P' \simeq \O\,{\widehat \otimes} \P$. As usual, we denote by $\{\P_j\}, \{\P'_j\}$ the filtrations of $\P, \P'$ defined as in Section \ref{annalg}, see also the beginning of Section \ref{notH}. We already know that $\I$ and $\P'_0$ both normalize each $\P'_n$. We also denote by $\P^k$ the graded component of $\P$ of degree $k\geq 2$.

Choose $\d \in \dd \subset \dd' \subset \dd \sd \P' =: \widetilde{\P'}$. Then the adjoint action of $\d$ on $\P'$ defines a derivation of $\P' \simeq \O\,{\widehat \otimes} \P$. Using Proposition \ref{derp} we understand that projecting $\Der \O\,{\widehat \otimes} \P$ on the first summand of the semidirect decomposition
$$\Der(\O\,{\widehat \otimes} \P) = (\Der \O \otimes \id_\P + \O \otimes \ad \E) \sd \O\,{\widehat \otimes} \ad \P$$
defines a Lie algebra homomorphism, so that we may find $\gamma(\d) \in \P'$ so that the adjoint action of the difference $\widehat{\d} = \d - \gamma(\d)\in \widetilde \P'$ on $\P'= \O  \,{\widehat \otimes} \P$ is induced by an element of $\Der \O$ acting on the first tensor factor plus a suitable $\O$-multiple of $\ad \E$. Notice that $\gamma(\d)$ is only determined modulo $Z(\P') = \O \otimes Z(\P)$, but we may choose it in a unique way if we demand that $\gamma(\d) \in \O \,{\widehat \otimes} \L_{-1}.$ Denote by $\widehat{\dd}$ the Lie subalgebra of $\widetilde{\L'}$ generated by elements $\widehat{\d}, \d\in \dd$. Notice that this is \underline{not} $\kk$-linearly generated by the above elements.

\begin{proposition}
The adjoint action of $\widehat{\dd}$ on $\P'$ normalizes $\P$ and centralizes $\P^0$. The map $\dd \ni \d \mapsto [\widehat{\d}] \in \widehat{\dd}/(\widehat{\dd} \cap Z(\P'))$ is a Lie algebra isomorphism.
\end{proposition}

\begin{proposition}
The Lie algebra 
$\N' = \widehat{\dd} + \P^{\geq 0} + \I + Z(\P')$ 
normalizes $\P'_n$ for every $n>0$.
\end{proposition}
\begin{proof}
We know that $\P'_0, \I, \widehat{\dd}$ and $Z(\P')$ all normalize $\P'_n$, so that their sum does so too.
However, from $\P' = \P + \I$ follows $\P'_0 \subset \P_0 + \I$. As the filtration on $\P$ induced by the grading coincide with our standard filtration, see Section \ref{nonontrivial}, we also obtain $\P_0 = \P^{\geq 0}$. We conclude that $\N' = \widehat{\dd} + \P'_0 + \I + Z(\P') = \widehat{\dd} + \P^{\geq 0} + \I + \O \,{\widehat\otimes}\, Z(\P)$ normalizes $\P'_n$.
Notice that $\P_0 = \P^{\geq 0} \subset \P'$ is a subalgebra and that $\I, Z(\P')$ are ideals, so the sum $\P^{\geq 0} + \I + Z(\P')$ is a subalgebra of $\P'$. Also, $\widehat{\dd}$ is a subalgebra of $\widehat \P'$ and normalizes $\P^{\geq 0} + \I + Z(\P')$, so that $\N'$ is a subalgebra of $\widetilde{\P'}$.
\end{proof}

\begin{lemma}\label{nprimedecomposition}
The sum
$$\N' = (\widehat{\dd} + \P^0 + (\Span_\kk\langle 1, x^1, \dots, x^r\rangle \otimes Z(\P)) \oplus (\P^{>0} + \m \,{\widehat\otimes}\, \P^\perp + \m^2 \,{\widehat\otimes}\, Z(\P))$$
is a direct sum decomposition as vector subspaces. Moreover, the second summand on the right-hand side is an ideal of $\N'$. The quotient $\N'/(\P^{>0} + \m \,{\widehat\otimes}\, \P^\perp + \m^2 \,{\widehat\otimes}\, Z(\P))$ is thus isomorphic to
$$(\widehat{\dd} + \P^0 + \O \,{\widehat\otimes}\, Z(\P))/(\m^2 \,{\widehat\otimes}\, Z(\P)).$$
\end{lemma}
\begin{proof}
Equality follows from $\P^{\geq 0} = \P^0 + \P^{>0}$ and
$\I = \m \,{\widehat\otimes}\, \P = \m \,{\widehat\otimes}\,(\P^\perp + Z(\P))$
by noticing that $Z(\P') = \O \,{\widehat\otimes}\, Z(\P)$ and $\O = \Span_\kk\langle 1, x^1, \dots, x^r\rangle + \m^2$.

The second summand on the right-hand side is a subalgebra of $\P'$ as $\P^\perp = \P^{\geq -1}$. It is an ideal of $\N'$ since $\P^{>0}, \P^\perp, Z(\P)$ are all homogeneous for the grading, and are thus $\widehat \dd + \P^0$-stable.
\end{proof}

\begin{remark}\label{overline}
Recall that $\widehat{\dd} = {\mathfrak D} + [{\mathfrak D}, {\mathfrak D}]$, where ${\mathfrak D} = \Span_\kk\langle \overline{\d}, \d \in \dd\rangle$. The Lie algebra $\widehat{\dd}$ is thus a finite-dimensional subalgebra of $\dd + \O \,{\widehat\otimes}\, Z(\P) \subset \dd \sd \P' \subset \dd' \sd \P' =: \widetilde{\P'}$. We know that $\widehat{\dd}$ centralizes $\P^0$, so that $\widehat{\dd} + \P^0 = \widehat{\dd} \oplus \P^0$ is a direct sum of Lie algebras. However, $\widehat{\dd}$ does not necessarily centralize $\P' = \O \,{\widehat\otimes}\, \P$, though it normalizes each $\m^j \,{\widehat\otimes}\, \P_k$.
\end{remark}

\begin{lemma}
Let $L = \Hd, L' = \Cur_H^{H'} L$, and $\widetilde{\P} = \dd \sd \P, \widetilde{\P'} = \dd' \sd \P'$ be the corresponding extended annihilation algebras. Then the quotient Lie algebra
$$\N'_k := (\widehat{\dd} + \P^0 + \O \,{\widehat\otimes}\, Z(\P))/(\m^k \,{\widehat\otimes}\,Z(\P))$$
is isomorphic to
\begin{itemize}
\item[---] $\dd \oplus \spdd$ when $k = 0$;
\item[---] $\dd_+ \oplus \spdd$ when $k = 1$,
\end{itemize}
Thus, $\N'_2$ is an abelian extension of both $\N'_0$ and $\N'_1$, 
and the intersection $[\N'_2, \N'_2]\cap (\O/\m^2) \,{\widehat\otimes}\, Z(\P)$ is only contained in $(\m/\m^2) \,{\widehat\otimes}\, Z(\P)$ when $\chi = 0$ and $\omega$ is exact.
\end{lemma}
\begin{proof}
Case $k=0$: First of all, the sum $\widehat{\dd} + \P^0$ is a direct sum of Lie algebras, as $\widehat{\dd}$ centralizes $\P^0$. Its intersection with $\O \,{\widehat\otimes}\,Z(\P)$ may only lie in $\widehat{\dd}$, as $Z(\P)$ lies in degree $-2$, and is therefore trivial. The statement then follows from Remark \ref{overline}.
Case $k=1$ follows from Proposition \ref{dactonxprime}\,{\em (iii)}. The last claim is then immediate.
\end{proof}

The following consequence of the Cartan-Jacobson theorem is well known.
\begin{lemma}\label{repabext}
Let 
\begin{equation}\label{abext}
0 \to \fa \stackrel{i}{\to} \fe \stackrel{\pi}{\to} \fg \to 0
\end{equation}
be an extension of the Lie algebra $\fg$ by the abelian ideal $\fa$. If $(V, \rho)$ is a finite-dimensional irreducible representation of $\fe$, then there exists a splitting (as vector spaces) $s: \fg \to \fe$ such that $\rho \circ s: \fg \to \gl(V)$ is a Lie algebra homomorphism.
\end{lemma}
\begin{proof}

By the Cartan-Jacobson theorem, each element from $a \in i(\fa)$ act on $V$ via multiplication by a scalar $\xi(a)$, thus yielding a linear form $\xi: \fa \to \kk$ whose kernel $\ker \xi = \fa \cap \ker \rho$ has codimension at most $1$ in $\fa$. If we set $\overline{\fa} = \fa/\ker \xi, \overline{\fe} = \fe/\ker \xi,$ then $\overline{\fe}$ is an extension of $\fg$ by the abelian ideal $\overline{\fa}$, and $\rho$ factors through $\overline{\fe}$ with a nontrivial action of $\overline{\fa}$ on $V$.

If $\overline{\fa} = 0$, or equivalently $\xi = 0$, then choose any section $s: \fg \to \overline{\fa}$ as vector spaces. Then $[s(g), s(h)] = s([g,h]) \mod \fa$ for every $g, h \in \fg$, and as $\fa \subset \ker \rho$ we obtain that $\rho\circ s: \fg \to \gl(V)$ is a Lie algebra homomorphism. 

We may thus assume without loss of generality that $\xi \neq 0$ and $\dim\overline{\fa} = 1$. Nonzero elements in $\overline{\fa}$ act via multiplication by nontrivial scalars on $V$, so they cannot be commutators in $\overline{\fe}$;
this forces the adjoint action of $\overline{\fe}$ on $\overline{\fa}$ to be trivial, so that $\overline{\fe}$ is a central extension of $\fg$, and $[\overline{\fe}, \overline{\fe}]$ to intersect $\overline{\fa}$ trivially, so that the central extension is not irreducible, and must admit a Lie algebra splitting $\overline s: \fg \to \overline{\fe}$.
Any lifting $s: \fg \to \fe$ of $\overline{s}$ will then satisfy $[s(g), s(h)] = s([g,h]) \mod \ker \xi$, whence $\phi = \rho \circ s: \fg \to \gl(V)$ is a Lie algebra homomorphism.
\end{proof}

\begin{remark}
If $s: \fg \to \fe$ is a splitting of \eqref{abext}, then every $g \in \fe$ may be decomposed as
$$g = s(\pi(g)) + (g - s(\pi(g))),$$
where $g - s(\pi(g)) \in i(\fa)$. When $\phi = \rho\circ s: \fg \to \gl(V)$ is a Lie algebra homomorphism, then 
$$\rho(g) = \phi(\pi(g)) + \xi(i^{-1}(g - s(\pi(g))))\id_V.$$
In other words, each finite-dimensional irreducible representation $V$ of an abelian extension $\fe$ of $\fg$ may only differ from an opportune irreducible action of $\fg$ on $V$ by multiples of $\id_V$.
\end{remark}

\subsection{Finite dimensional irreducible representations of $\N'/\P_n$, $n\geq 0$}

Our setting is as in Section 7. $L = \Hd$, $L' = \Cur_H^{H'}$ and $\widetilde \P, \widetilde \P'$ are the corresponding extended annihilation algebras. $V'$ is an irreducible representation of $L'$, so that the $L'$-action on $V'$ is non-trivial and $\ker_n V' = \{v \in V'\,|\, \P'_n v = 0\}$ is a nonzero finite-dimensional vector subspace of $V'$ when for sufficiently large values of $n$. Then $\N'$ stabilizes $\ker_n V'$, as $\N'$ normalizes $\P'_n$. If $R$ is a nonzero $\N'$-submodule of $\ker_n V'$, then $\Ind_{\N'}^{\widetilde \P'} R$ projects to $V'$. Without loss of generality, we may assume $R$ to be $\N'$-irreducible. The following fact is then going to prove useful.

\begin{lemma}
Let $R$ be an irreducible finite-dimensional $\N'$-module with a trivial action of $\P'_n$. Then $\P^{>0} + \m \,{\widehat\otimes}\, \P^\perp + \m^2 \,{\widehat\otimes}\, Z(\P)$ acts trivially on $R$.
\end{lemma}
\begin{proof}
The descending central series of all three summands eventually sits inside $\P'_n$, so that their sum is an ideal, by Lemma \ref{nprimedecomposition}, projecting to the radical of $\N'/\P'_n$. However, the adjoint action of $\P^0$ on the first two summands decomposes as a sum of nontrivial irreducible representations, whereas the third summand lies in $[\m \,{\widehat\otimes}\, \P^\perp, \m \,{\widehat\otimes}\, \P^\perp]$. We may then use \cite[Lemma 3.4]{BDK1} to conclude.
\end{proof}

\begin{proposition}
Let $R$ be an irreducible finite-dimensional $\N'$-module with a trivial action of $\P'_n$. Then the action of $\N'$ on $R$ factors through
$$\N'/(\P^{>0} + \m \,{\widehat\otimes}\, \P^\perp + \m^2 \,{\widehat\otimes}\, Z(\P)) \simeq (\widehat{\dd} \oplus \P^0) + (\O/\m^2)  \,{\widehat\otimes}\, Z(\P) = \N'_2.$$
If $\rho: \N'_2 \to \gl(R)$ denotes the above action, one may find an irreducible action $\phi: \N'_1 \to \gl(R)$ such that
$\rho(x) - \phi(x + \m \,{\widehat\otimes}\, Z(\P'))$ is a scalar multiple of $\id_R$ for every $x \in \N'_2$.
\end{proposition}

We may summarize the last proposition as follows: if $V'$ is a finite irreducible representation of the Lie pseudoalgebra $L'$, then one may find a finite dimensional irreducible $\N'$-module $R$ such that $V'$ is a quotient of $\Ind_{\N'}^{\widetilde \P'} R$, and the action of $\N'$ on $R$ is uniquely described by that of $\overline{\dd} \oplus \P^0$, along with a suitable scalar action of the abelian ideal $(\m/\m^2) \,{\widehat\otimes}\, Z(\P)$, which however does not affect irreducibility of $R$.

As $\widetilde \P' = \dd' + \N'$, then $\Ind_{\N'}^{\widetilde \P'} R$ is 
isomorphic to $\ue(\dd') \otimes R$, where the $\dd'$-action is obtained by left multiplication on the first tensor factor.

\begin{theorem}
Let $V'$ be a finite irreducible representation of the current Lie pseudoalgebra $L' = \Cur_H^{H'}\Hd$. Then there exists a finite-dimensional irreducible $\dd_+ \oplus \spdd$-module $R = \Pi \boxtimes U$ such that $V'$ is a quotient of the left $H'$-module $H' \otimes R$, endowed with the pseudoaction
\begin{align*}
e*(1 \otimes v) & = \sum_{ij} (\overline{\d_i} \overline{\d_j} \otimes 1) \otimes_{H'} (1 \otimes f^{ij}.v)\\
& - \sum_k (\overline{\d_k} \otimes 1) \otimes_{H'} (1 \otimes (\d^k + \ad^\sp(\d^k)).v - \d^k \otimes v)\\
& + (1 \otimes 1) \otimes_{H'} (1 \otimes c.v)\\
& + (t \otimes 1) \otimes_{H'} (1 \otimes v).
\end{align*}
where $v \in R$ and either $t = 0$ or $t \in \dd' \setminus \dd$.
\end{theorem}
\begin{proof}
Let $R \subset \ker_1 V'$ be a (finite-dimensional) irreducible $\N'$-summand. Then the $\N'$-action on $R$ factors through the quotient $\N'_2$, and one may use Lemma \ref{repabext} to find a section of $\N'_1$ in $\N'_2$ in order to express the above $\N'_2$-action on $R$ by means of an irreducible action of $\N'_1 \simeq \dd_+ \oplus \spdd$, at least up to adding multiples of $\id_R$. One may then proceed similarly to \cite[Proposition 6.3]{BDK3}, while expressing the action of $\N'_2$ on $v \in R$ as the sum of the $\N'_1$-action and scalar multiples of $v$. The extra terms that arise, with respect to the ordinary tensor module pseudoaction, are of two kinds: first, there are terms of the form
\begin{equation}
\sum_{i = 1}^r (S(\d_i) \otimes 1) \otimes_{H'} (1 \otimes (x^i\otimes c).v),
\end{equation}
where $c = e^{-\chi} \otimes_H e \in \P \in Z(\P)$ is the linear generator of the center of $\P$, that occur because of the possibly nontrivial (scalar) action of $(\m/\m^2) \otimes Z(\P)$ on $R$. Secondly, there are extra scalar terms which originate from expressing the $\N'_2$-action via an irreducible $\N'_1$-action. Adding things up yields a total contribution of the form
$$(t \otimes 1) \otimes_{H'} (1 \otimes v),$$
where $t$ is an element of $\dd'$ whose projection to the vector space quotient $\dd'/\dd$ provides a complete description of the central action of $(\m/\m^2) \otimes Z(\P)$ by means of a linear functional on $\m/\m^2 \simeq (\dd'/\dd)^*$.

Notice that when $\m \,{\widehat \otimes}\, Z(\P)$ acts trivially, then both extra terms vanish and $t = 0$. If instead the $\m \,{\widehat \otimes}\, Z(\P)$ is nontrivial, then the second summation lies in $(\dd \otimes 1) \otimes_{H'} (1 \otimes v)$, and cannot cancel with the first contribution.
\end{proof}

We denote the above $L'$-module by $\VHtR$. 

\begin{remark}
If $t = 0$, then $\V_{\chi, \omega, 0, \dd'}^{\textrm H}(R) \simeq \Cur_H^{H'} \VH(R)$, where $\VH(R)$ is as in Section \ref{hdtm}. When $t \neq 0$, then $\ad_\chi$ preserves $\dd$ and lies in $\spdd$. In particular $\t := \kk t + \dd$ is a Lie subalgebra of $\dd'$ strictly containing $\dd$, and $\VHtR = \Cur_{\ue(\t)}^{H'} \V_{\chi, \omega, t, \t}^{\textrm H}(R)$.
\end{remark}

\section{Singular vectors and irreducibility}

Let $L$ be one of the finite primitive simple Lie pseudoalgebras listed in Section \ref{sprimitive}, and consider a finite representation $V'$ of the current Lie pseudoalgebra $L' = \Cur_H^{H'} L$.  We know that $\ker_n V' \subset V'$ is nonzero for sufficiently large values of $n$ and is stabilized by the action of $\N'= \widehat \dd + \L^{\geq 0} + \I \subset \widetilde \L'$. Moreover, every finite-dimensional irreducible $\N'$-submodule of $\ker_n V'$ lies in $\ker_1 V'$ and has a trivial action of the ideal $\Z = \L^{>0} + \m \otimes \L^{\perp} + \m^k \otimes \L$, where $k = 2$ when $L = \Hd$ and $k = 1$ otherwise. If $L \neq \Hd$, then $\N'/\Z\simeq \widehat \dd \oplus \L^0$ is isomorphic to
\begin{itemize}
\item[---] $\dd \oplus \gld$ when $L = \Wd$;
\item[---] $\dd \oplus \sld$ when $L = \Sd$;
\item[---] $\dd \oplus \sp(\ker \theta, \di \theta)$ when $L = \Kd$.
\end{itemize}
When $L = \Hd$, instead, $\N'/\Z$ is an abelian extension of $\dd_+ \oplus \spdd$.

\begin{definition}
Let $V'$ be a representation of a simple current Lie pseudoalgebra $L'$. 
An element $v \in V'$ is a {\em singular vector} if $\Z.v = 0$. The set of all singular vectors $\sing V':= \{v \in V'\,|\, \Z.v = 0\} \subset \ker_1 V'$ is a subspace of $V'$.
\end{definition}
We have already showed that $\sing V'$ is a finite-dimensional $\N'$-submodule of $V'$ as soon as $V'$ is a {\em finite} representation of $L$ with a nontrivial action, e.g., when $V'$ is a finite irreducible $L'$-module.

\begin{proposition}\label{singaction}
Let $L$ be a primitive Lie pseudoalgebra, and $S\subset L, \ell\in \NN,$ as in Section \ref{sprimitive}. If $V'$ is a finite Lie pseudoalgebra representation of $L' = \Cur_H^{H'} L$, then
$$(1 \otimes_H S)*(\sing V') \subset (\fil^\ell H \otimes \kk) \otimes_{H'} (\fil^1 H)\cdot (\sing V') + \sum_{k = 1}^r (\d_i \otimes 1) \otimes_{H'} (\sing V'),$$
and the second summand is possibly nonzero only when $L = \Hd$.
\end{proposition}
\begin{proof}
Follows by using Theorem \ref{repdescription} with a PBW basis of $H'$ corresponding to a basis $\{\d_1, \dots, \d_{N+r}\}$ of $\dd'$ such that $\d_{r+1}, \dots, \d_{N+r}$ is a basis of $\dd$.
\end{proof}

We have already seen that $\Ind_{\N'}^{\widetilde \L'} R \simeq H' \otimes R$ is a finite $L'$-module projecting to $V'$. It is (isomorphic to) a current representation $\Cur_H^{H'} \V(R)$ when the second summand vanishes, and to $\V_{t, \dd'}(R), t \in \dd' \setminus \dd$, otherwise.

\subsection{Singular vectors in current modules}

Whenever $V$ is an $L$-module, we set $\ker V: = \{ v \in V \,|\, L*v = 0\}$. It is clearly an $H$-submodule of $V$. Here we compute singular vectors in the current representation $\Cur_H^{H'}V$ of $\Cur_H^{H'}L$. It is convenient to treat the case $L = \Hd$ separately.
\begin{proposition}\label{wsk}
Let $L \neq \Hd$ be a primitive Lie pseudoalgebra as in Section \ref{sprimitive}. If $V$ is an $L$-module, then $\sing \Cur_H^{H'} V = 1 \otimes_H \sing V + H' \otimes_H \ker V$.
\end{proposition}
\begin{proof}
Let $u = \sum_{K\in\NN^r \times 0}  \d^{(K)} \otimes_H u_K \in H' \otimes_H V$. If $a*u_K = \sum_{L \in 0 \times \NN^{N}} (\d^{(L)} \otimes 1) \otimes_H u_{K+L}$, then
$$(1 \otimes_H a) * u = \sum_{K, L} (\d^{(L)} \otimes \d^{(K)}) \otimes_{H'} u_{K+L}.$$
Using Proposition \ref{singaction} and Corollary \ref{RSuniqueness}, we argue that if $u$ is singular, then $u_{K+L} = 0$ whenever $K \neq 0$ or, equivalently, $u_K \in \ker V$; similarly, $u_0 \in \sing V$.
\end{proof}

We may now step on to the case $L = \Hd$.

\begin{lemma}
Let $V$ be a Lie pseudoalgebra representation of $L = \Hd = He$. Then $C(V):= \{v \in V\,|\, e*v \in (1 \otimes 1) \otimes_H V\}$ equals $\ker V$.
\end{lemma}
\begin{proof}
Clearly, $C(V) \subset \sing V$, as $u\in V$ is singular precisely when $e*u \in (\fil^2 H \otimes \kk) \otimes_H V$. Then $u \in C(V)$ implies $\spdd.u = 0$ and $\d u - \d.u = 0$ for all $\d\in \dd$, where $\d.u$ denotes the action of $\d \in \dd_+$ on $u \in \sing V$. However $\d.u = \d u$ for every $\d \in \dd$ implies
$$([\d, \d'] + \omega(\d \wedge \d')c).u = [\d, \d']_+ .u = \d.(\d'.u) - \d'.(\d.u) = \d(\d'u) - \d'(\d u) = [\d, \d']u,$$
that is $\omega(\d\wedge\d')c.u = 0$ for all $\d, \d' \in \dd$. As $\omega$ is nondegenerate, this forces $c.u = 0$, whence $e*u = 0$.
\end{proof}

\begin{proposition}
Let $V$ be a Lie pseudoalgebra representation of $\Hd$. Then $\sing \Cur_H^{H'} V = 1 \otimes_H \sing V + H' \otimes_H \ker V$.
\end{proposition}
\begin{proof}
We shall prove $\sing \Cur_H^{H'} V = 1 \otimes_H \sing V + \dd' \otimes_H C(V) + H' \otimes_H \ker V$. Then the claim follows using the previous lemma.
First of all, Proposition \ref{singaction} implies
$$(1 \otimes_H e)*(\sing V') \subset ((\fil^2 H \otimes \kk) + (\kk \otimes \dd')) \otimes_{H'} V'.$$
Now proceed as in Proposition \ref{wsk} and compute $(1 \otimes_H e)*u$. If $\{\epsilon_i, 1 \leq i \leq r\}$ denotes the standard canonical basis of $\NN^r$, then Corollary \ref{RSuniqueness} yields $u_{\epsilon_i} \in C(V)$, and $a* u_K = 0$ if $K \neq 0, \epsilon_i$, whence $u_K \in \ker V$. Similarly, $u_0 \in \sing V$.
\end{proof}

\begin{corollary}
Let $V$ be a finite irreducible module over the primitive Lie pseudoalgebra $L$. Then $\sing \Cur_H^{H'} V = 1 \otimes_H \sing V$.
\end{corollary}
\begin{proof}
Follows from $\ker V = 0$.
\end{proof}

\begin{remark}
Recall that the embedding $\iota: H \to H'$ is a pure ring homomorphism, so that 
$m \mapsto 1 \otimes_H m$ provides an injective $H$-linear homomorphism $M \to \Cur_H^{H'} M$ for every left $H$-module $M$. Thus each Lie $H$-pseudoalgebra embeds $H$-linearly in the corresponding current Lie pseudoalgebras, and similarly for representations.

In particular, the subspace $1 \otimes_H \sing V$ in the previous corollary is isomorphic to $\sing V$, both as a vector space and as a $\dd \oplus \L^0$-module (resp. $\dd_+ \oplus \L^0$-module when $L$ is of type H).
\end{remark}

\subsection{Irreducibility of $\VHtR$}

Let $\VHtR, t \in \dd' \setminus \dd,$ be a Lie pseudoalgebra representation of $L' = \Cur_H^{H'} \Hd$ as introduced in Section \ref{Hdexceptional}. 
We will now investigate irreducibility of $\VHtR$ by explicitly computing its singular vectors.
\begin{theorem}
Let $\dd \subset \dd'$ be Lie algebras, and $t\in \dd' \setminus \dd$ an element satisfying the conditions of Proposition \ref{twisted}. Then $\sing \VHtR = \kk \otimes R$. As a consequence, the $L'$-module $\VHtR, t \in \dd' \setminus \dd,$ is irreducible as soon as $R$ is an irreducible $\dd_+ \oplus \spdd$-module.
\end{theorem}
\begin{proof}
The representation $\VHtR = H' \otimes R$ of $L' =\Cur_H^{H'}\Hd$ is defined by the Lie pseudoalgebra action
\begin{align}\label{tvsaction}
e *_t (1 \otimes v) & = \sum_{ij} (\overline{\d_i \d_j} \otimes 1) \otimes_{H'} (1 \otimes f^{ij}.v)\\
& - \sum_k (\overline{\d_k} \otimes 1) \otimes_{H'} (1 \otimes (\d^k + \ad^\sp \d^k).v - \d^k \otimes v)\nonumber\\
& + (1 \otimes 1) \otimes_H (1 \otimes c.v) + (t \otimes 1) \otimes_{H'} (1 \otimes v).\nonumber
\end{align}
Expression \eqref{tvsaction} may be right-straightened to
\begin{align}
e *_t (1 \otimes v) & = \sum_{i,j} (1 \otimes \d_i \d_j) \otimes_{H'} (1 \otimes f^{ij}.v)\\
& + \sum_k (1 \otimes \d_k) \otimes_{H'} (1 \otimes (\d^k + \ad^\sp(\d^k)).v - \d^k \otimes v)\nonumber\\
& - (1 \otimes t) \otimes_{H'} (1 \otimes v)\nonumber\\
& + \mbox{ terms in } (\kk \otimes \fil^1 H) \otimes_{H'} (\fil^1 H' \otimes V).\nonumber
\end{align}
Now pick a basis $t = \d_1, \dots, \d_{N+r}$ of $\dd'$ so that $\d_{r+1}, \dots, \d_{N+r}$ is a basis of $\dd$. Use the corresponding PBW basis to express
\begin{equation}\label{pbwu}
0 \neq u = \sum_{L \in \ZZ^{N+r}} \d^{(L)} \otimes u_L.
\end{equation}
Plugging \eqref{pbwu} into $e *_t u$ yields a right-straightened expression. If $n$ is the maximal value of $|L|$ such that $u_L \neq 0$, choose among all such $L = (l_1, l_2, \dots, l_{N+r})$ with $|L| = n$ one with the highest value of $l_1$. Then the term multiplying $1 \otimes t\d^{(L)}$ in $e *_t u$ equals $1 \otimes u_L$. 
If $u$ is singular and $L \neq 0$, this must vanish by Proposition \ref{singaction}, leading to a contradiction. We conclude that  $\sing \V_{t, \dd'} = \kk \otimes R$. The remaining claim follows easily.
\end{proof}

\section{Classification of irreducible modules over non primitive simple Lie pseudoalgebras}

We summarize all previous results in the following theorem
\begin{theorem}\label{main}
Let $\dd \subset \dd'$ be finite-dimensional Lie algebras, $H \subset H'$ their universal enveloping algebras endowed with the canonical cocommutative Hopf algebra structure. The following is a complete list of finite irreducible representations of the current Lie pseudoalgebra $L' = \Cur_H^{H'} L$, where $L$ is a primitive Lie pseudoalgebra:
\begin{itemize}
\item[---] $\Cur_H^{H'} V$, where $V$ is a finite irreducible $L$-module;
\item[---] $\VHtR$, where $L = \Hd$,  $R$ is a finite-dimensional irreducible representation of $\dd_+ \oplus \spdd$, and $t\in \dd' \setminus \dd$ satisfies 
\begin{itemize}
\item[{\em (i)}] $\ad_\chi t$ preserves $\dd$ and lies in $\spdd$;
\item[{\em (ii)}] $[s, t] = 0,$ where $s$ satisfies $\chi = \iota_s \omega$.
\end{itemize}
\end{itemize}
The only nontrivial isomorphisms between the above irreducible modules are those described in Theorem \ref{nontrivialiso}.
\end{theorem}
\begin{proof}
We are only left with proving that representations in the list are pairwise non-isomorphic. This is done by comparing the $\N'$-actions on the space of singular vectors contained in each representation.

First of all, $\sing \Cur_H^{H'} V = \sing V$, and irreducible representations of a simple primitive Lie pseudoalgebra are all told apart by their singular vectors, viewed as a $\dd \oplus \L^0$-module. This takes care of the cases $L = \Wd, \Sd, \Kd$.

When $L = \Hd$, no current irreducible representation is isomorphic to a  non-current one as the action of central elements lying in $\m \,{\widehat \otimes}\, Z(\L)$ is trivial in the former case and nontrivial in the latter. Furthermore, any isomorphism $\VHt(R_1) \simeq \VHt(R_2)$ induces an isomorphism of the corresponding singular vectors, viewed as a $\dd_+ \oplus \spdd$-module. However, singular vectors are all constant and lie in a single $\dd_+ \oplus \spdd$-component, which is isomorphic to $R_1, R_2,$ respectively. Thus, {\bf for equal values of $t$}, the $\Cur_H^{H'} \Hd$-modules $\VHt(R_1), \VHt(R_2)$ are isomorphic precisely when $R_1, R_2$ are isomorphic representations of $\dd_+ \oplus \spdd$.

Finally recall that the action on $\VHtR$ of central elements lying in $\m \,{\widehat \otimes}\, Z(\L)$ is via multiplication by scalars, and defines a linear functional $\m/\m^2 \to \kk$. However, $\m/\m^2 \simeq (\dd'/\dd)^*$ and this linear functional corresponds to a unique class in the vector space quotient $\dd'/\dd$, which coincides with $[t]$. This shows that if $\VHtR$ and $\VHtprime(R')$ are isomorphic, then elements $t, t'$ project to the same class of the quotient $\dd'/\dd$. Then their difference $\delta = t' - t\in \dd$ satisfies the conditions in Lemma \ref{equivalent}, so that Theorem \ref{nontrivialiso} yields
$\VHtprime(R') = \VHt(R' \otimes (\kk_{\pi^*(\iota_{\delta}\omega)} \boxtimes \kk))$, and one falls back to the previous case.
\end{proof}


\bibliographystyle{amsalpha}



\end{document}